\documentclass
[
    fontsize = 11 pt,
    american,
    captions = tableheading,
    numbers = noenddot,
    abstracton,
    paper = letter,
]
{scrartcl}

\usepackage[T1]{fontenc}
\usepackage[utf8]{inputenc}
\usepackage
[
    babel = true, % Enables language-specific tuning.
]
{microtype}           % Uses the text space more efficiently.
\usepackage{csquotes} % Uses the correct quotes according to the current language.

\usepackage[dvipsnames]{xcolor} % Allows it to define colors. The option says that common names can be used.

\definecolor{stroke1}{HTML}{2574A9} % This color is used as the standard color to highlight things.

\usepackage{amsmath}
\usepackage{amssymb}
\usepackage{amsthm}
\usepackage{thmtools}
\usepackage{mathtools}
\usepackage{thm-restate}
\usepackage{dsfont}        % Yields far better blackboard-bold letters than \mathbb. Use \mathds in order to write such letters.
\usepackage{braceMnSymbol} % Adjusts overbraces and underbraces such that longer versions are put together seamlessly.

\usepackage
[
    ttscale = 0.85, % Scales the typewriter font.
]
{libertine} % The main font used in this thesis.
\usepackage
[
    libertine,    % Changes the math font to libertine (the main font).
    slantedGreek, % Makes all greek letters italic by default. If you want to use an upright greek letter, use ›\up‹ immediately followed by the letter’s name. For example, \upGamma displays an upright uppercase gamma.
    vvarbb,       % Changes the \mathbb font to another font. However, \mathbb remains ugly and should not be used. Use \mathds instead.
    libaltvw,     % Uses different characters for v und w that look far better than the default ones.
]
{newtxmath} % The main math font of this thesis. It fits well with the main font.
\usepackage{url} % Responsible for URL formatting.
\usepackage{bm}  % Allows to use sensible bold letters in math mode. This package has to go after the font packages. Otherwise it does not work correctly!

\usepackage{graphicx} % The standard package for including graphics into your document.
\usepackage
[
    subrefformat = simple, % Formats the label of the \subref command without parentheses.
    labelformat = simple,  % Formats the mark of a subfigure without parentheses.
]
{subcaption}         % Enables it to have subfigures inside of a single figure.

\usepackage{array}     % Improves the way that tables can be formatted.
\usepackage{booktabs}  % Adds lines (called ›rules‹) that can be used in tables and improves spacing.
\usepackage{longtable} % Allows to make tables that span multiple pages.
\usepackage{pdflscape} % Allows to change a page into landscape. This is handy if a table is very wide.

\usepackage[shortlabels]{enumitem} % Adds tons of useful features to enumeration environments.

\usepackage
[
    ruled,         % Creates lines at the top and at the bottom. Further, the caption is now above the algorithm.
    vlined,        % Shows the scope of a statement spanning multiple lines via a small vertical bar. Thus, no closing tags are needed.
    linesnumbered, % Shows line numbers.
]
{algorithm2e} % Allows to write pseudocode.

\usepackage{xspace}   % Adds the functionality that a space after a command will be shown as a space in the output.
\usepackage
[
    shortcuts, % Allows to use short symbols for non-breaking hyphens and dashes instead of lengthy commands.
]
{extdash}             % Adds non-breaking hyphens and dashes.
\usepackage{setspace} % Allows to easily chnage the spacing inside of the document.

\usepackage{xparse}    % Is used in order to define reasonable commands.
\usepackage{footnote}  % Allows it to extend the environments footnotes can be used in. It is said that this package is in conflict with ›hyperref‹. I did not note any troubles. However, if something is fishy, it is probably best to not use this package.
\usepackage{afterpage} % Adds the \afterpage command, which specifies that the provided argument shall be processed after the current page is finished.
\usepackage
[
    textsize = scriptsize, % Determines the text size of the TODO note.
]
{todonotes}            % Adds TODO notes to the document. These are small text areas inside of the margin of a page.

\usepackage
[
    sortcites,              % Sort multiple references when citing them together.
    style = alphabetic,     % The style of a citation mark.
    defernumbers,           % Makes sure that references always have unique numbers. This is important if you use multiple bibliographies.
    safeinputenc,           % Allows to use UTF8 characters in the bibliography and tries to translate them into TeX automatically.
    backref = true,         % Creates back references in the bibliography.
    backrefstyle = three,   % Compresses three or more consecutive pages in the back references into a range.
    hyperref = true,        % Makes links generated by biblatex clickable. If hyperref is not used, a warning is issued.
    maxbibnames = 99,       % The maximum number of names displayed in the bibliography.
    maxcitenames = 3,       % The maximum number of names displayed when using commands like ›textcite‹. The default is 3. After that, ›et al.‹ is used.
    minalphanames = 3,       % The minimum number of names to display for the alphabetic sorting style.
]
{biblatex} % Used in order to format the bibliography.

\DefineBibliographyStrings{english}%
{%
    backrefpage  = {\lowercase{s}ee page}, % For a single page number.
    backrefpages = {\lowercase{s}ee pages} % For multiple page numbers.
}

\usepackage
[
    bookmarks = true,                 % Generates boodmarks for the PDF.
    bookmarksopen = false,            % The bookmarks are closed by default.
    bookmarksnumbered = true,         % The bookmarks use the numbers of the corresponding headline.
    pdfstartpage = 1,                 % The first page seen when opening the PDF.
    colorlinks = true,                % The text of hyperlinks is colored instead of having a colored box around it.
    allcolors = stroke1,              % Every hyperlink uses the same color. If you want to change specific colors, use the commands below.
]
{hyperref} % The standard package that is used for creating hyperlinks inside of a document.

\usepackage
[
    noabbrev,   % The words in front of the labels are not abbreviated.
    nameinlink, % Extends the link of a reference to the word in front of it.
]
{cleveref} % This package must be included after ›hyperref‹. It creates clever references that know what they refer to.
\usepackage{multirow}
\usepackage{caption}
\usepackage{lineno}

\date{}

\makeatletter
    \def\IfEmptyTF#1%
    {%
        \if\relax\detokenize{#1}\relax%
            \expandafter\@firstoftwo%
        \else%
            \expandafter\@secondoftwo%
        \fi%
    }
\makeatother

\NewDocumentCommand{\mathOrText}{m}
{%
    \ensuremath{#1}\xspace%
}

\let\originalleft\left
\let\originalright\right
\renewcommand{\left}{\mathopen{}\mathclose\bgroup\originalleft}
\renewcommand{\right}{\aftergroup\egroup\originalright}

\makeatletter
    \DeclareRobustCommand{\bfseries}%
    {%
        \not@math@alphabet\bfseries\mathbf%
        \fontseries\bfdefault\selectfont%
        \boldmath%
    }

\xspaceaddexceptions{]\}}

\allowdisplaybreaks

\crefname{ineq}{inequality}{inequalities}
\creflabelformat{ineq}{#2{\upshape(#1)}#3}

\crefname{term}{term}{terms}
\creflabelformat{term}{#2{\upshape(#1)}#3}

\crefname{cond}{condition}{conditions}
\creflabelformat{cond}{#2{\upshape(#1)}#3}

\crefname{assume}{assumption}{assumptions}
\creflabelformat{assume}{#2{\upshape(#1)}#3}

\deffootnote[1.2 em]{1.2 em}{0 em}{\makebox[1.2 em][l]{\textbf{\thefootnotemark}}}

\let\oldfootnote\footnote

\newlength{\spaceBeforeFootnote} % Denotes the space before the footnote mark in em.
\newlength{\spaceAfterFootnote}  % Denotes the space after the footnote mark in em.

\RenewDocumentCommand{\footnote}{o o o m}%
{%
    \IfNoValueTF{#1}%
    {%
        \oldfootnote{#4}%
    }%
    {%
        \setlength{\spaceBeforeFootnote}{\IfEmptyTF{#1}{0}{#1} em}%
        \IfNoValueTF{#2}%
        {%
            \hspace*{\spaceBeforeFootnote}\oldfootnote{#4}%
        }%
        {%
            \setlength{\spaceAfterFootnote}{\IfEmptyTF{#2}{0}{#2} em}%
            \hspace*{\spaceBeforeFootnote}\IfNoValueTF{#3}{\oldfootnote{#4}}{\oldfootnote[#3]{#4}}\hspace*{\spaceAfterFootnote}%
        }%
    }%
}

\makesavenoteenv{figure}
\makesavenoteenv{table}
\makesavenoteenv{tabular}

\declaretheoremstyle
[
   	spaceabove = \topsep,
   	spacebelow = \topsep,
   	headfont = \bfseries,
   	headformat = \textcolor{stroke1}{$\blacktriangleright$} \NAME~\NUMBER \NOTE,
   	notefont = \bfseries,
   	notebraces = {(}{)},
   	bodyfont = \normalfont,
   	postheadspace = 0.5 em,
   	qed = \textcolor{stroke1}{\bfseries$\blacktriangleleft$},
]
{myTheoremStyle}

\declaretheorem
[
   	style = myTheoremStyle,
   	name = Corollary,
    sharenumber = conjecture,
]
{corollary}
\declaretheorem
[
   	style = myTheoremStyle,
   	name = Theorem,
    sharenumber = conjecture,
]
{theorem}
\declaretheorem
[
   	style = myTheoremStyle,
   	name = Definition,
    sharenumber = conjecture,
]
{definition}

\NewDocumentCommand{\functionTemplate}{m m m m o}%
{%
    \IfNoValueTF{#5}%
    {%
        \mathOrText{#1\left#2{#4}\right#3}%
    }%
    {%
        \mathOrText{#1#5#2{#4}#5#3}%
    }%
}

\newcommand*{\leftBracketType}{(}
\newcommand*{\rightBracketType}{)}

\NewDocumentCommand{\createFunction}{m m o o}%
{%
    \renewcommand*{\leftBracketType}{\IfNoValueTF{#3}{(}{#3}}%
    \renewcommand*{\rightBracketType}{\IfNoValueTF{#4}{)}{#4}}%
    \NewDocumentCommand{#1}{o o}%
    {%
        \IfNoValueTF{##1}%
        {%
            \mathOrText{#2}%
        }%
        {%
            \functionTemplate{#2}{\leftBracketType}{\rightBracketType}{##1}[##2]%
        }%
    }%
}

\DeclareDocumentCommand{\probabilisticFunctionTemplate}{m m O{} o}
{%
    \functionTemplate{#1}%
    {\lbrack}%
    {\rbrack}%
    {#2\IfEmptyTF{#3}{}{\ \IfNoValueTF{#4}{\left}{#4}\vert\ \vphantom{#2}#3\IfNoValueTF{#4}{\right.}{}}}%
    [#4]%
}

\newcommand*{\N}{\mathOrText{\mathds{N}}}

\newcommand*{\R}{\mathOrText{\mathds{R}}}

\newcommand*{\indicatorFunctionSymbol}{\mathds{1}}

\RenewDocumentCommand{\Pr}{m O{} o}%
{%
    \probabilisticFunctionTemplate{\mathrm{Pr}}{#1}[#2][#3]%
}

\NewDocumentCommand{\E}{m O{} o}%
{%
    \probabilisticFunctionTemplate{\mathds{E}}{#1}[#2][#3]%
}

\NewDocumentCommand{\Var}{m O{} o}%
{%
    \probabilisticFunctionTemplate{\mathrm{Var}}{#1}[#2][#3]%
}

\DeclareDocumentCommand{\bigO}{m o}%
{%
    \functionTemplate{\mathrm{O}}{(}{)}{#1}[#2]%
}

\DeclareDocumentCommand{\smallO}{m o}%
{%
    \functionTemplate{\mathrm{o}}{(}{)}{#1}[#2]%
}

\DeclareDocumentCommand{\bigTheta}{m o}%
{%
    \functionTemplate{\upTheta}{(}{)}{#1}[#2]%
}

\DeclareDocumentCommand{\bigOmega}{m o}%
{%
    \functionTemplate{\upOmega}{(}{)}{#1}[#2]%
}

\DeclareDocumentCommand{\smallOmega}{m o}%
{%
    \functionTemplate{\upomega}{(}{)}{#1}[#2]%
}

\DeclareDocumentCommand{\eulerE}{o}%
{%
    \mathOrText{\mathrm{e}\IfNoValueTF{#1}{}{^{#1}}}%
}

\DeclareDocumentCommand{\poly}{m o}%
{%
    \functionTemplate{\mathrm{poly}}{(}{)}{#1}[#2]%
}

\createFunction{\id}{\mathrm{id}}

\NewDocumentCommand{\ind}{m o o}%
{%
    \IfNoValueTF{#2}%
    {%
        \mathOrText{\indicatorFunctionSymbol_{#1}}%
    }%
    {%
        \functionTemplate{\indicatorFunctionSymbol_{#1}}{(}{)}{#2}[#3]%
    }%
}

\DeclareDocumentCommand{\dom}{m o}%
{%
    \functionTemplate{\mathrm{dom}}{(}{)}{#1}[#2]%
}

\DeclareDocumentCommand{\rng}{m o}%
{%
    \functionTemplate{\mathrm{rng}}{(}{)}{#1}[#2]%
}

\DeclareDocumentCommand{\d}{o}%
{%
    \mathrm{d}\IfNoValueTF{#1}{}{^{#1}}%
}

\DeclareDocumentCommand{\set}{m m o}%
{%
    \mathOrText{\IfNoValueTF{#3}{\left}{#3}\{#1\ \IfNoValueTF{#3}{\left}{#3}\vert\ \vphantom{#1}#2\IfNoValueTF{#3}{\right.}{}\IfNoValueTF{#3}{\right}{#3}\}}%
}

\DeclareDocumentCommand{\randomProcess}{m o}
{%
    \mathOrText{X^{(#1)}\IfNoValueTF{#2}{}{_{#2}}}%
}

\DeclareDocumentCommand{\transformedProcess}{o}
{%
    \mathOrText{Y\IfNoValueTF{#1}{}{_{#1}}}%
}

\DeclareDocumentCommand{\filtration}{o}
{%
    \mathOrText{\mathcal{F}\IfNoValueTF{#1}{}{_{#1}}}%
}

\newcommand*{\cSIRS}{\mathOrText{\textrm{cSIRS}}}
\newcommand*{\placy}{\mathOrText{\textrm{labeled cSIRS}}}
\newcommand*{\timePoint}{\mathOrText{t}}

\newcommand*{\numberOfVertices}{\mathOrText{n}}

\newcommand*{\infectionRate}{\mathOrText{\lambda}}
\newcommand*{\infectionConstant}{\mathOrText{c}}

\newcommand*{\deimmunizationRate}{\mathOrText{\varrho}}
\newcommand*{\decayRate}{\mathOrText{\alpha}}
\newcommand*{\contactProcess}{\mathOrText{C}}
\newcommand*{\contactProcessProjection}{\mathOrText{C'}}
\newcommand*{\infectedDiscrete}[1]{\mathOrText{I_{\timeContinuous{#1}}}}

\newcommand*{\recoveredDiscrete}[1]{\mathOrText{R_{\timeContinuous{#1}}}}

\newcommand*{\timeDiscrete}{\mathOrText{t}}
\newcommand*{\timeContinuous}[1]{\mathOrText{\tau_{#1}}}

\newcommand*{\degree}{\mathOrText{d}}
\newcommand*{\degreeGap}{\mathOrText{\varepsilon_d}}
\newcommand*{\spectralGap}{\mathOrText{\delta}}

\DeclareDocumentCommand{\lyapunovHelper}{o}
{%
    \mathOrText{f\IfNoValueTF{#1}{}{(#1)}}%
}

\DeclareDocumentCommand{\lyapunovFunction}{m o}
{%
    \mathOrText{F\IfNoValueTF{#2}{}{(#2)}}%
}

\DeclareDocumentCommand{\potentialFunctionIMinusR}{o}
{%
    \mathOrText{H\IfNoValueTF{#1}{}{(#1)}}%
}
\newcommand*{\potentialIMinusR}[1]{\mathOrText{H_{#1}}}

\newcommand*{\infectedContinuous}[1]{\mathOrText{I_{#1}}}
\newcommand*{\susceptibleContinuous}[1]{\mathOrText{S_{#1}}}
\newcommand*{\recoveredContinuous}[1]{\mathOrText{R_{#1}}}
\newcommand*{\infectedSet}[1]{\mathOrText{I_{#1}'}}
\newcommand*{\susceptibleSet}[1]{\mathOrText{S_{#1}'}}
\newcommand*{\recoveredSet}[1]{\mathOrText{R_{#1}'}}

\title{Gradually Declining Immunity Retains the Exponential Duration of Immunity-Free Diffusion}

\author{%
    Andreas Göbel$^{*}$ \and Nicolas Klodt$^{*}$ \and Martin~S. Krejca$^{\dagger}$ \and Marcus Pappik$^{*}$
}

\publishers
 {
  	\footnotesize $^{*}$ Hasso Plattner Institute, University of Potsdam, Potsdam, Germany\\
  	\{tobias.friedrich, andreas.goebel, nicolas.klodt, marcus.pappik\}@hpi.de \\[1 ex]

  	$^{\dagger}$ LIX, CNRS, Ecole Polytechnique, Institut Polytechnique de Paris, Palaiseau, France \\
  	martin.krejca@polytechnique.edu
 }

\begin{document}

\maketitle

\begin{abstract}
  Diffusion processes pervade numerous areas of AI, abstractly modeling the dynamics of exchanging, oftentimes volatile, information in networks.
  A central question is how long the information remains in the network, known as \emph{survival time}.
  For the commonly studied SIS process, the expected survival time is at least super-polynomial in the network size already on star graphs, for a wide range of parameters.
  In contrast, the expected survival time of the SIRS process, which introduces temporary immunity, is always at most polynomial on stars and only known to be super-polynomial for far denser networks, such as expanders.
  However, this result relies on featuring \emph{full} temporary immunity, which is not always present in actual processes.
  We introduce the \emph{\cSIRS} process, which incorporates \emph{gradually declining} immunity such that the expected immunity at each point in time is identical to that of the SIRS process.
  We study the survival time of the \cSIRS process rigorously on star graphs and expanders and show that its expected survival time is very similar to that of the SIS process, which features no immunity.
  This suggests that featuring gradually declining immunity is almost as having none at all.
\end{abstract}

\section{Introduction}
\label{sec:introduction}

Diffusion processes on graphs are prevalent in various domains of AI research, modeling a broad range of applications, such as information diffusion~\cite{SunCM2023SocialInfluenceMaximization,JiangRF2023PoliticalLearning,LiuRGCXTL2023DOVID19,SunRZLY2022HypergraphAttentionNetwork,RazaqueRKAR2022Survey,SharmaHSL2021FakeNews}, rumor spreading~\cite{kempe2003maximizing}, infections~\cite{Pastor-SatorrasCVMV15Survey,leskovec2007cost}, and computer viruses~\cite{berger2005spread,BorgsCGS10Antidote}.
These processes usually share the same core mechanics, which are naturally expressed as extensions or variations of the well-known \emph{SI process} from epidemiology (see \cite{Pastor-SatorrasCVMV15Survey} for an extensive survey). The SI process is defined as a Markov chain on an underlying graph with vertices that are either \emph{infected} or \emph{susceptible} to an infection, and the infection spreads randomly over edges with an \emph{infection rate}~$\infectionRate$.

A very important variant of the SI process studied extensively both empirically and theoretically, e.g., \cite{PhysRevE.86.041125,ferreira2016collective,NamNS22SISinfinite,BorgsCGS10Antidote,ganesh2005effect} is the \emph{SIS} process (also commonly called \emph{contact process}), which allows infected vertices to become susceptible again.
In turn, susceptible vertices can become infected again, and a fundamental question is how long an infection survives on a graph.
This (random) time period is called the \emph{survival time} of the process, and it is closely tied to the expansion properties of the graph~\cite{ganesh2005effect}.
Most importantly, the SIS process exhibits a super-polynomial expected survival time (also called \emph{endemic}) already on star graphs with~$n$ leaves once the infection rate is at least in the order of only~$n^{-1/2 + \varepsilon}$, with $\varepsilon > 0$ being an arbitrary constant.
This result translates to any graph that contains a star as a subgraph, implying that the SIS process goes endemic on many natural networks, such as scale-free networks, as they contain large stars as subgraphs \cite{berger2005spread}.

One potential way to counteract endemic behavior is to introduce immunity against the infection into the system.
This is classically modeled with the \emph{SIRS} process, which introduces a \emph{recovered} state, in which vertices are immune to the infection.
Different from the SIS process, infected vertices transition now randomly into the recovered state, from which they transition into the susceptible state, based on a random rate~$\deimmunizationRate$ called the \emph{deimmunization rate}.
While there is a plethora of empirical results on the SIRS process, e.g., \cite{Wang_2017,PhysRevLett.86.2909,ferreira2016collective}, most of the theoretical results use some simplifying assumptions such as mean-field approaches~\cite{prakash2012threshold,Bancal10}.
To the best of our knowledge, the recent paper by \cite{friedrich2024irrelevance} is the first fully rigorous paper on this process.
In their work, the authors prove that the SIRS process has an at most polynomial expected survival time on stars for \emph{any} infection rate~$\infectionRate$ if the deimmunization rate~$\deimmunizationRate$ is constant. This result has recently been complemented by \cite{lam2024optimalboundsurvivaltime} with a matching lower bound.
This strongly contrasts the SIS process and shows that immunity can be beneficial in fighting back the infection.
\cite{friedrich2024irrelevance} furthermore show that the SIRS process becomes endemic on expander graphs, that is, dense graphs.
This suggests that the benefit of immunity degrades with the density of the graph.

The polynomial expected survival time of the SIRS process on stars by \cite{friedrich2024irrelevance} is a consequence from the fact that vertices that recover from the infection are \emph{fully} immune to re-infection until they spontaneously become susceptible. A main reason for this assumption in the SIRS process is that the resulting process is a time-homogeneous Markov chain, which simplifies the analysis. However, for a lot of real-life processes, this is a strong simplification. In infectious diseases, like SARS-CoV-2 for example, the number of antibodies drops continuously, which suggests that the resistance drops over time instead of just vanishing at some point~\cite{sanderson2021covid,wheatley2021evolution,gaebler2021evolution,vaccines10010064}. Similar phenomena occur in several social phenomena~\cite{ZhangLZLZZ16Overview}, for example, innovation adoption, where people adopt an innovative product, such as a phone, and do not require a new one until a certain time passes, after which the people get more receptive to buying a new product. In this example, the declining immunity models that a person can buy a new phone before the old one breaks if the new features are sufficiently better, which becomes more and more likely the more time passes.

In a newly proposed model, \cite{watve2024epidemiology} aim to capture the complexity of continuously declining immunity. The authors show via simulations that this model explains real-world phenomena well, such as multiple small outbreaks of an infection. However, this model features a multitude of parameters and is too complex to rigorously analyze its effects.

A more promising approach for a rigorous analysis is a line of research that models the declining immunity while keeping the process time-homogeneous \cite{aguas2006pertussis,fouchet2008visiting}. To this end, a new semi-recovered state is added, in which the immunity has partially worn off. In this state, vertices get infected at a smaller rate and go upon infection into a state of mild infection. Thus, this setting only provides a coarse, discretized version of declining immunity.
The analysis is done via simulations and mean-field theory.

While there have been studies of variants of declining immunity incorporated into infection models, to our knowledge, none of them include rigorous mathematical analyses.

\subsection{Main Contribution}

We study the impact of gradually degrading immunity when compared to the temporarily full immunity of the SIRS process.
To this end, we introduce and mathematically rigorously analyze the \emph{\cSIRS} process, which incorporates the gradual decline of immunity into the SIRS process such that the expected degree of immunity is the same as in the original SIRS process.
We study the expected survival time of the \cSIRS process on star and expander graphs, allowing us to compare our results to the SIS and the SIRS process.
Moreover, our lower bounds on the expected survival time hold for any graph that contains a star or expander as subgraph (\Cref{obs:monotone} and \Cref{cor:expander}, respectively).
\Cref{table:results} summarizes our results on stars.

We observe that while the definition of the SIRS and the \cSIRS process may seem similar, they behave fundamentally differently.
Although both processes have an at most logarithmic expected survival time for sufficiently small infection rates $\infectionRate \in \bigO{n^{-1/2}}$, the expected survival time of the \cSIRS process jumps immediately to a super-polynomial expected survival time for only slightly larger values of $\infectionRate \in \bigTheta{n^{-1/2 + \varepsilon}}$, with $\varepsilon \in \R_{>0}$ being an arbitrary constant, whereas the expected survival time of the SIRS process remains polynomial, regardless of~$\infectionRate$.
Thus, the \cSIRS process behaves far more closely to the SIS than to the SIRS process.
This shows that although the probability to become re-infected is the same in the SIRS and \cSIRS process for the first re-infection attempt, the gradual decline of immunity in the \cSIRS process has a dramatic impact on its survival time.
This impact is more comparable to a process that does not feature any immunity, although our lower bounds for the \cSIRS process are lower than for the SIS process, hinting at a still existing, albeit far less impactful benefit of immunity.
The fundamental reason for this different behavior is that many infection attempts in the \cSIRS process challenge the immunity repeatedly whereas this is not the case in the SIRS process.
Overall, our results suggest that incorporating immunity only has a substantial benefit if it can be guaranteed \emph{fully} for a sufficient amount of time.
Since our results carry over to supergraphs and since stars are present in many graphs, our results cover a wide range of different graph classes, such as scale-free graphs~\cite{berger2005spread,friedrich2024irrelevance}.

\begin{table}
  \centering
  \begin{tabular}{p{7 em}p{3.9 em}p{3.9 em}p{3.9 em}}
    \toprule
    infection rate                                                                                                    & SIS                                    & SIRS                                                                            & \cSIRS                                                                    \\
    \midrule
    $\infectionRate \in \bigO{\numberOfVertices^{-1/2}}$                                                              & $\bigO{\log(\numberOfVertices)}$       & $\bigO{\log(\numberOfVertices)}$                                                & $\bigO{\log(\numberOfVertices)}$                                          \\
    $\infectionRate \in \bigTheta{\numberOfVertices^{-1/2+\varepsilon}}$ \hspace*{0.5 em} and $\infectionRate \leq 1$ & $\bigOmega{\eulerE^{n^{\varepsilon}}}$ & $\widehat{\Theta}\left(n^{2\varepsilon\varrho}\right)$ Thm.~\ref{thm:SIRStight} & $\bigOmega{\eulerE^{n^{2\varepsilon/3}}}[\big]$ Thm.~\ref{thm:cSIRSlower} \\
    $\infectionRate >1$                                                                                               & $\bigOmega{\eulerE^{n^{1/2}}}[\big]$   & $\widehat{\Theta}(n^{\varrho})$                                                 & $\bigOmega{\eulerE^{n^{1/3}}}[\big]$                                      \\
    \bottomrule
  \end{tabular}
  \caption{Expected survival time $\E{T}$ of different processes on stars with $n$ leaves, starting with an infected center and no recovered leaves. The parameter $\lambda$ is the infection rate of the process, $\varrho$ the deimmunization rate of the SIRS process, assumed to be constant, and $\varepsilon$ any constant in $(0,1/2]$. The results from this paper have their theorem numbers below them. The logarithmic upper bounds and SIS results follow from~\protect\cite{ganesh2005effect}.
    The hat in the big-O notation means we omit sub-polynomial factors.
    The results in the last row follow from \Cref{obs:monotone} and~\protect\cite{friedrich2024irrelevance}.
  }
  \label{table:results}
\end{table}

Our results for the SIRS process on stars are an improvement over those by \cite{friedrich2024irrelevance}, who only proved an upper bound of $\widehat{\mathrm{O}}(n^\deimmunizationRate)$ for all values of~$\infectionRate$.
Our bounds are tight up to sub-polynomial factors and showcase different regimes for larger values of~$\infectionRate$. We note that \cite{lam2024optimalboundsurvivaltime} investigated the same question and independently found similar results that do not need the sub-polynomial factors.

In addition to our results on stars, we prove that the at least exponential survival time of the SIRS process with sufficiently large infection rate on expanders carries over to the \cSIRS process (\Cref{cor:expander}).
Hence, once the graph is sufficiently dense, the SIRS and the \cSIRS process start acting similarly.

From a mathematical perspective, the analysis of the \cSIRS process proves challenging, as it is not Markovian.
We approach this problem by introducing an intermediate process, where we relabel some of the vertices in the \cSIRS process such that the resulting process resembles a SIRS process with extra transitions that allow recovered vertices to become directly infected (\Cref{fig:states}, bottom right).
Due to its closer relation to the SIRS process, this intermediate process allows an easier analysis based on more traditional methods.

\section{Preliminaries}
\label{sec:preliminaries}

We define the general notation for infection processes as well as the SIRS process, following the notation by \cite{FrGoKlKrPa2024}.
The \cSIRS process is defined in \Cref{sec:cSIRS}.

Infection processes are random processes on labeled graphs where the dynamics are solely driven by Poisson processes and the labels of the vertices.
Poisson processes are one-dimensional Poisson point processes that output a random subset of the non-negative real numbers. We consider infection processes on finite, undirected graphs with~$n$ vertices. All big-O notation concerns asymptotics in this value of~$n$. Especially, a \emph{constant} is a value independent of~$n$.

The SIRS process is defined over a graph $G=(V,E)$ and two values $\infectionRate, \deimmunizationRate \in \R_{>0}$, which are the \emph{infection rate} and \emph{deimmunization rate}, respectively. To each edge $e \in E$, we assign a Poisson process~$M_e$ of rate~$\infectionRate$, and to each vertex $v \in V$, we assign two Poisson processes: $N_v$ with rate~$1$ and $O_v$ with rate~$\deimmunizationRate$. We call these processes \emph{clocks}, and when a time point $\timePoint \in \R_{\geq 0}$ is part of a clock's output, we say that the clock \emph{triggers} at~$\timePoint$.
We assume that all clocks evolve simultaneously and independently.
Since all clocks are Poisson processes, there is almost surely no point at which two clocks trigger at once, and each clock outputs almost surely a countably infinite number of triggers such that for each point, there exists a trigger that is at least as large.
Let $\{\gamma_i\}_{i\in\N_{\geq0}}$ with $\gamma_0=0$ denote the (random) sequence of all these triggers.

A SIRS process is a random process $(C_t)_{t \in \R_{\geq 0}}$ that partitions~$V$ for all time points $\timePoint \in \R_{\geq 0}$ into the set~$\susceptibleSet{\timePoint}$ of \emph{susceptible} vertices, the set~$\infectedSet{\timePoint}$ of \emph{infected} vertices, and the set~$\recoveredSet{\timePoint}$ of \emph{recovered} vertices, that is, $C_t = (\susceptibleSet{\timePoint}, \infectedSet{\timePoint}, \recoveredSet{\timePoint})$.
The value of~$C_0$ is given, and all other values are defined inductively based on $\{\gamma_i\}_{i\in\N_{\geq0}}$ such that the process is for all $i \in \N_{\geq 0}$ constant on $[\gamma_i, \gamma_{i+1})$.
That is, states only change when a clock triggers, especially, for all $t \in [0, \gamma_1)$, it holds that $\susceptibleSet{\timePoint} = \susceptibleSet{0}$, $\infectedSet{\timePoint} = \infectedSet{0}$, and $\recoveredSet{\timePoint} = \recoveredSet{0}$.
Depending on which clock triggers and the state of the involved vertices, we have the following transitions for all $i \in \N_{\geq 0}$ and any $s \in [\gamma_i, \gamma_{i + 1})$:
\begin{itemize}
  \item \textbf{Susceptible to infected.}
        Let $e = \{u, v\} \in E$ with $\gamma_{i + 1} \in M_e$ and $u \in \infectedSet{s}$ as well as $v \in \susceptibleSet{s}$.
        Then for all $t \in [\gamma_{i + 1}, \gamma_{i + 2})$ holds that $u, v \in \infectedSet{t}$.
        We say that \emph{$u$ infects~$v$} (at time~$\gamma_{i + 1}$).
  \item \textbf{Infected to recovered.}
        Let $v \in V$ with $\gamma_{i + 1} \in N_v$ and $v \in \infectedSet{s}$.
        Then for all $t \in [\gamma_{i + 1}, \gamma_{i + 2})$ holds that $v \in \recoveredSet{t}$.
        We say that \emph{$v$ recovers} (at time~$\gamma_{i + 1}$).
  \item \textbf{Recovered to susceptible.}
        Let $v \in V$ with $\gamma_{i + 1} \in O_v$ and $v \in \recoveredSet{s}$.
        Then for all $t \in [\gamma_{i + 1}, \gamma_{i + 2})$ holds that $v \in \susceptibleSet{t}$.
        We say that \emph{$v$ becomes susceptible} (at time~$\gamma_{i + 1}$).
\end{itemize}
In addition, we may call vertices that are not infected \emph{healthy}, and we may call the transition of an infected vertex to a non-infected state\footnote{In the SIRS process, this is the transition to the recovered state. In the SIS process, this is the transition to the susceptible state.} \emph{healing}.

During each of the transitions above, all vertices not mentioned remain in their respective set.
Moreover, note that not all triggers lead necessarily to a state change.
For example, if the clock of an edge triggers whose two incident vertices are already infected, nothing changes.

As states remain unchanged for most of the time, we consider in our analyses only those time points where a state change occurs.
Formally, we consider $\{\gamma_0\} \cup \{\gamma_i \mid i \in \N_{>0} \land \contactProcess_{\gamma_{i}} \neq \contactProcess_{\gamma_{i-1}}\}$, which we index by the increasing sequence $\{\timeContinuous{i}\}_{i \in \N_{\geq 0}}$.
For all $i \in \N$, we call $\timeContinuous{i}$ the $i$-th \emph{step} of the SIRS process.

We are interested in the first point in time where no vertex is infected, as such a state leads quickly to the (only) stable state where all vertices are susceptible.
We call $T \coloneqq \inf \{\timePoint \in \R_{\geq 0} \mid \infectedSet{\timePoint} = \emptyset\}$ the \emph{survival time} of the SIRS process, and we say that the infection \emph{dies out} or \emph{goes extinct} at~$T$.

\paragraph{Useful definitions and mathematical tools.}
For the graphs we consider, it is sufficient to only consider the number of vertices in each of the sets of the partition of~$V$. To this end, we define for all $\timePoint \in \R_{\geq0}$ the random variables $\susceptibleContinuous{\timePoint} = |\susceptibleSet{\timePoint}|$, $\infectedContinuous{\timePoint} = |\infectedSet{\timePoint}|$, and $\recoveredContinuous{\timePoint} = |\recoveredSet{\timePoint}|$. These random variables change based on events triggered by the clocks. We say an event \emph{happens at a rate of $r \in \R_{> 0}$} if and only if the set of clocks causing this event by triggering has a sum of rates equal to~$r$.

We use stochastic domination to bound the values of random processes with other random processes. We say that a random process $(X_t)_{t \in \R_{\geq 0}}$ \emph{dominates} another random process $(Y_t)_{t \in \R_{\geq 0}}$ if and only if there exists a coupling $(X'_t, Y'_t)_{t \in \R_{\geq 0}}$ such that for all $t \in \R_{\geq 0}$ holds $X'_t \geq Y'_t$.

One way we use domination is to connect our processes to well analyzed processes like the gambler's ruin process. For this process, we require the following known results.

\begin{theorem}[Gambler's ruin~{\protect\cite[page~345]{feller1957introduction}}]\label{pre:gamblersRuin}
  Let $(P_t)_{t \in \N}$ be the amount of money that a player has in a gambler's ruin game that has a probability of $p \neq 1/2$ for them to win in each step. Let $q=1-p$. The game ends at time~$T$ when the player either reaches the lower bound $\ell$ or the upper bound $u$ of money. Then
  \begin{enumerate}
    \item $\Pr{P_T = \ell} = \frac{1-(p/q)^{u-P_0}}{1-(p/q)^{u-\ell}}$;
    \item $\Pr{P_T = u} = \frac{1-(q/p)^{P_0-l}}{1-(q/p)^{u-\ell}}$.\qedhere
  \end{enumerate}
\end{theorem}

In the analyses, terms like $\prod_{i=1}^{n}{\frac{i}{i+c}}$ show up. We bound their asymptotic behavior with the following theorem, which follows from result~(1) of \cite{tricomi1951asymptotic}.

\begin{theorem}[Ratio of Gamma Functions~{\protect\cite[page~133]{tricomi1951asymptotic}}]\label{thm:gamma}
  Let $n\in \N$, and let $\alpha,\beta \in \R_{>0}$ be constants. Then
  $
    \frac{\Gamma (n+\alpha)}{\Gamma (n+\beta)} \in \bigTheta{n^{\alpha - \beta}}$.
\end{theorem}

This yields the following corollary.

\begin{corollary}\label{pre:gamma}
  Let $n,m\in \N$, and let $c \in \R_{>0}$ be a constant. Then
  $  \prod_{i=m}^{n}{\frac{i}{i+c}} \in \bigTheta{\frac{m^c}{n^c}}$.
\end{corollary}

\begin{proof}
  The result follows from \Cref{thm:gamma} and the fact that $\prod_{i=m}^{n}{\frac{i}{i+c}} = (\prod_{i=1}^{n}{\frac{i}{i+c}}) / (\prod_{i=1}^{m - 1}{\frac{i}{i+c}})$ and $\prod_{i=1}^{n}{\frac{i}{i+c}} = \frac{\Gamma(c+1)\cdot \Gamma(n+1)}{\Gamma(n+c+1)}$.
\end{proof}

\section{The \cSIRS Process}
\label{sec:cSIRS}

\begin{figure}
  \centering
  \includegraphics[width = \linewidth]{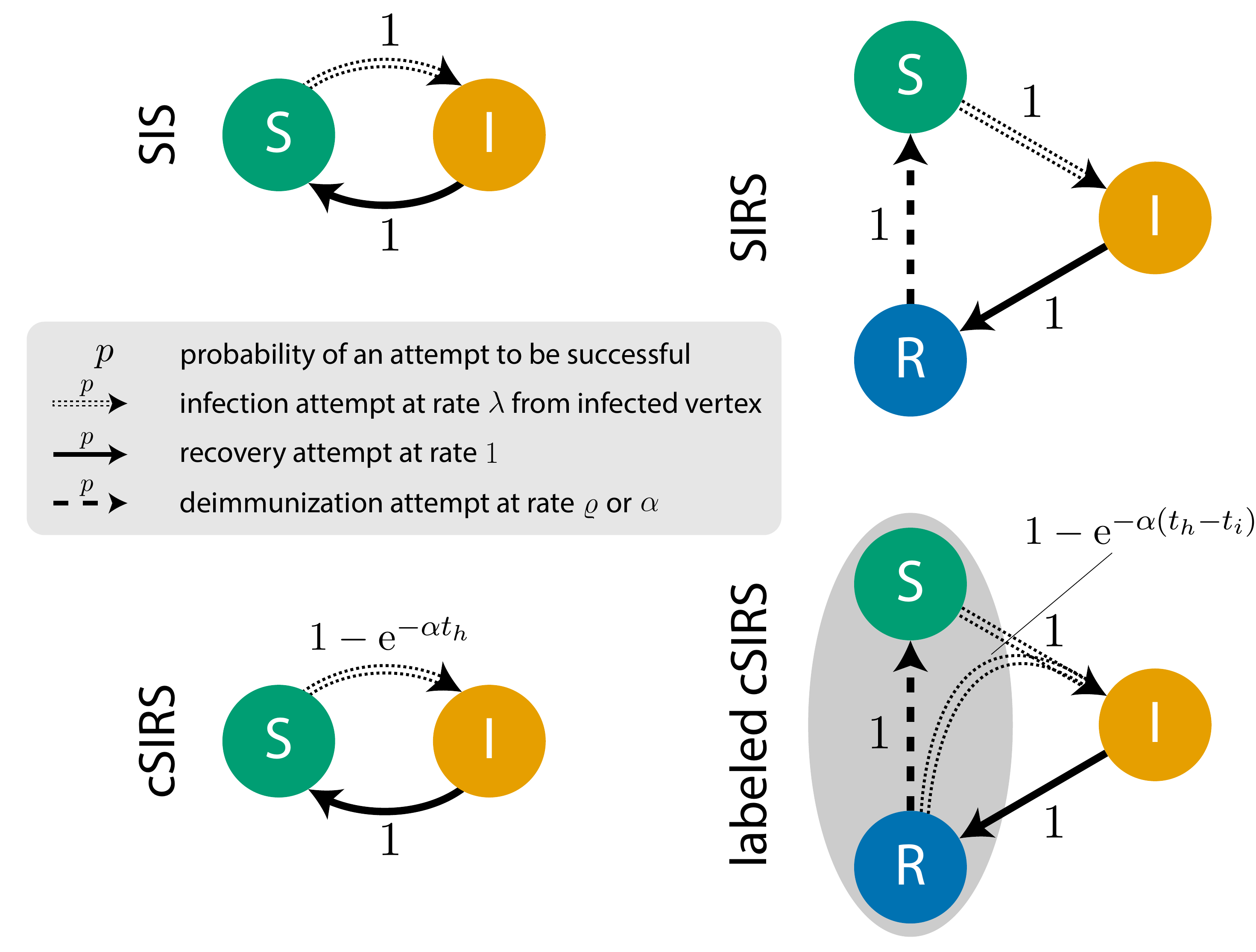}
  \caption{State transitions of a vertex in the shown processes. Vertices are susceptible (S), infected (I), or recovered (R). Edges represent the existence of a Poisson clock that triggers a transition attempt with a rate dependent on the arrow type. Note that edges to infected vertices represent one clock for each infected neighbor. The numbers on the arrows represent the probability of a successful attempt. We use~$t_h$ to denote the time passing since the vertex last healed, and~$t_i$ to denote the time passing since the last infection attempt after the vertex healed (or the last time the vertex healed, whichever is smaller).}
  \label{fig:states}
\end{figure}

We introduce the \cSIRS (continuous SIRS) process, which aims to model the continuous decay of immunity. It behaves mostly like the SIRS process, but instead of a recovered vertex being fully immune and losing this immunity after an exponentially distributed time, the immunity decreases exponentially in a deterministic way. More precisely, the process behaves like an SIS process where every vertex~$v$ additionally has a resistance $r_v \in [0,1]$, which is initialized with~$0$. Whenever a vertex changes its state from infected to susceptible, its resistance is set to 1. The resistance then exponentially declines with some \emph{resistance decay rate} $\decayRate \in \R_{>0}$. That is, after a time $t \in \R_{\geq 0}$ after healing, the resistance is $\eulerE^{-\decayRate t}$. When a vertex $v$ would get infected by an infection clock on an edge in the SIS process, it now gets infected with probability $1-r_v$ and otherwise remains susceptible. A depiction of these transitions is shown in \Cref{fig:states}, bottom left.

We note that the SIRS process is expressible in a similar manner. It is obtained by dropping the resistance from~$1$ to~$0$ after an exponentially distributed random time instantly instead of letting it decline gradually. By choosing the same~$\decayRate$ in the \cSIRS process as~$\deimmunizationRate$ in the SIRS process, the expected resistance in the SIRS process matches the actual resistance in the \cSIRS process at all times. Note that the \cSIRS process is not Markovian anymore, which removes some useful properties that are normally used to analyze processes. Below, we explain how we still mange to extend existing results to the new process.

\subsection{Useful Properties}

As a first observation, we note that the number of infected vertices in the \cSIRS process is dominated by the number of infected vertices in the SIS process with the same parameters. This means that all upper bounds on survival times in the SIS process carry over to the \cSIRS process.

\begin{restatable}{theorem}{cSIRSvSISdominance}
  \label{thm:cSIRSvSISdominance}
  Let $G$ be a graph and let $\infectionRate \in \R_{>0}$. Let $C$ be a cSIRS process on $G$ with infection rate $\infectionRate$ and a resistance decay rate $\decayRate \in \R_{>0}$, and let $C'$ be a SIS process on $G$ with infection rate $\infectionRate$ that starts with the same infected vertices as~$C$. Then there exists a coupling of~$C$ and~$C'$ such that the set of infected vertices of~$C$ is for all points in time a subset of the set of infected vertices of~$C'$.
\end{restatable}

\begin{proof}
  In order to show the domination, we couple the two processes and show that in this coupling at all times, all vertices that are infected in $C$ are also infected in $C'$. The coupling is the trivial coupling in which all healing clocks and infection clocks trigger at the same time in both processes.

  We prove our claim via induction on the time steps in which one of the two processes changes state to show that all infected vertices in $C$ are also infected in $C'$. Recall that the set of events is almost surely countable and that we defined an ordering on them. Both processes start with the same set of infected vertices, so the claim holds in the base case. Now assume the statement holds at time step $t \in \N$. To show that the statement holds at time steps $t+1$, we only have to consider the vertices that changed state. Note that almost surely only one vertex changes in each time step. If the change is a healing clock, by the coupling, the vertex is susceptible after the change in both processes, so the statement is true. If the change is an infection clock, a vertex $v$ is infected by an infected vertex $u$ in one of the two processes. By the induction hypothesis, $u$ is therefore definitely infected in $C'$ as $u$ being infected in $C$ implies it being infected in $C'$ as well. Therefore $v$ is infected in $C'$ at time step $t+1$ and the induction statement is true.
\end{proof}

To obtain lower bounds on the survival time, we modify the \cSIRS process to be closer to the SIRS process. This makes it easier to apply previous results to this new process and to talk about the SIRS and the new process using the same notation. To this end, we define the \emph{\placy} process. It is almost equivalent to the \cSIRS process with the only difference that we extend the SIRS process instead of the SIS process. As it is just a SIRS process with an extra rate to infect recovered vertices, results from the SIRS process are much easier to adapt to this definition. The success probability for infection attempts is chosen in a way such that relabeling all recovered vertices in the \placy process to susceptible yields the \cSIRS process. The definition is visualized in \Cref{fig:states}.

\begin{definition}[\placy]
  A \placy process on a graph $G$ with infection rate $\infectionRate \in \R_{>0}$ and resistance decay rate $\decayRate \in \R_{>0}$ is defined like a SIRS process with $\deimmunizationRate = \decayRate$, with the difference that recovered vertices have a possibility to become infected. For all $\timePoint \in \R_{\geq 0}$ and each vertex~$v$ that is recovered at time~$\timePoint$, let~$t_h$ be the time that passed from the last time that~$v$ recovered, that is, for $t^* \coloneqq \sup\{s \in \R_{\geq 0} \mid s \leq \timePoint \land v \textrm{ recovers at } s\}$, let $t_h = \timePoint - t^*$, where we define $\sup \emptyset = 0$.
  Moreover, let~$t_i$ be the time that passed since the last infection attempt involving~$v$, or let $t_i = t_h$ if no such attempt occurred since~$t^*$. That is, let $t_i = \timePoint - \max\big\{t^*, \sup\{s \in \R_{\geq 0} \mid s < \timePoint \land \exists \{u, v\} \in E\colon (s \in M_{\{u, v\}} \land u \textrm{ is infected at } s)\}\big\}$. Then each infection attempt at~$v$ is successful with probability $1-\eulerE^{-\decayRate (t_h-t_i)}$.
\end{definition}

We note that we believe that adding an extra rate to infect recovered vertices directly should increase the survival time of the infection. We show in \Cref{cor:expander} that this believe is correct for expander graphs. However, to the best of our knowledge, there is no general result for the SIRS process that proves this belief. There are some scenarios where infecting a recovered vertex leads to it being recovered instead of susceptible later which could potentially block a relevant infection later.
Thus, we argue differently in the following by showing that the \cSIRS process is equivalent to the \placy in which all recovered vertices are relabeled to be susceptible.

\begin{restatable}{observation}{PLACY}
  \label{obs:PLACY}
  A \cSIRS process and a \placy process with the same parameters (including the same initialization) can be coupled in a way such that at each time they have exactly the same set of infected vertices. Especially, they have the same distribution of survival times.
\end{restatable}

\begin{proof}
  We couple the processes by letting all healing and infection clocks trigger at the same time. Moreover, an infection in the \cSIRS process is successful if and only if the same trigger is successful in the \placy process. We show that this is a valid coupling, that is, one whose transition probabilities remain truthful to the original processes.

  By definition, infected vertices heal at a rate of 1 in both processes and attempt to infect non-infected vertices at a rate of~$\lambda$. It remains to show that all infection attempts are successful with the same probability in both processes. To this end, consider an infection attempt on a non-infected vertex $v$ in both processes. We make a case distinction on whether there was a failed infection attempt since the last time~$v$ healed.

  If there was not, then $t_h = t_i$ and thus the probability to transition in the \placy process directly from recovered to infected is~$0$, by definition. Hence, the attempt is only successful in the \placy process if the vertex transitioned into the susceptible state in the time window since it recovered. The probability for this happening is $1-\eulerE^{-\decayRate t_h}$, which matches the probability for a successful infection in the \cSIRS process.

  If there was a failed infection attempt before, the vertex was still recovered at that time in the \placy process. In the $t_i$ time window, the vertex transitions into the susceptible state with probability $p= 1-\eulerE^{-\decayRate t_i}$. The next infection attempt in the \placy process is successful if the vertex transitions before into the susceptible state or if the attempt on the recovered vertex is successful. Hence, for the success probability $q$ of the infection attempt holds
  \begin{align*}
    q & = (1-\eulerE^{-\decayRate t_i})\cdot 1 + \eulerE^{-\decayRate t_i} \cdot (1-\eulerE^{-\decayRate(t_h-t_i)}) \\
      & = 1-\eulerE^{-\decayRate t_i} + \eulerE^{-\decayRate t_i} - \eulerE^{-\decayRate t_h}                       \\
      & = 1 - \eulerE^{-\decayRate t_h}.
  \end{align*}

  As this matches the probability of a successful infection attempt in the \cSIRS process again, the coupling stated at the beginning is correct.
\end{proof}

\section{\cSIRS and SIRS on Stars}

It is known that the SIRS process never survives super-polynomially long on stars when the deimmunization rate is constant~\cite{FrGoKlKrPa2024}. This is in big contrast to the SIS process, in which there is a relatively tight threshold at which the survival time goes from logarithmic to super-polynomial (see also \Cref{table:results}).
We aim to see how the \cSIRS process compares to these results.
To this end, we analyze the expected survival time of both the \cSIRS process (\Cref{thm:cSIRSlower}) and the SIRS process (\Cref{thm:SIRStight}) at this threshold---the latter, since the only existing result so far~\cite{FrGoKlKrPa2024} is only an upper bound.
We note that the lower bounds of our two main theorems above hold for \emph{any} graph that contains a star as a subgraph, as long as this subgraph satisfies the starting conditions mentioned in the theorems.
As various graph classes contain large stars, such as scale-free graphs, this typically translates into bounds based on the overall graph size, not just the star size \cite{berger2005spread}.

For the SIRS process, we show in \Cref{thm:SIRSlower,thm:SIRSupper} almost tight polynomial upper and lower bounds for the survival time above the threshold, showing that the upper bound by \cite{FrGoKlKrPa2024} is almost tight.
Our two theorems directly imply the following theorem.

\begin{restatable}{theorem}{SIRStight}
  \label{thm:SIRStight}
  Let $G$ be a star with $\numberOfVertices$ leaves. Let \contactProcess be a SIRS process on $G$ with infection rate~$\infectionRate\leq 1$ with $\infectionRate \in \Omega(n^{-1/2})$ and with constant deimmunization rate \deimmunizationRate that starts with infected center and no recovered leaves. Let $T$ be the survival time of $C$. Then $\E{T}  \in \widehat{\Theta}\left((\infectionRate^2\numberOfVertices)^{\deimmunizationRate}\right)$.
\end{restatable}

For the \cSIRS process, we show a behavior very similar to the SIS process.
That is, from the same value of the infection rate onward as in the SIS process, the \cSIRS process exhibits a super-polynomial expected survival time.

\begin{restatable}{theorem}{cSIRSlower}
  \label{thm:cSIRSlower}
  Let $G$ be a star with $\numberOfVertices$ leaves. Let $\varepsilon \in (0,1/2]$ be a constant and let \contactProcess be a \placy process or \cSIRS process on $G$ with infection rate~$\infectionRate\leq 1$ with $\infectionRate \in \bigOmega{n^{-1/2+\varepsilon}}$ and with constant resistance decay rate \decayRate that starts with infected center and no recovered leaves. Let~$T$ be the survival time of $C$. Then $\E{T}  \in \bigOmega{\eulerE^{\numberOfVertices^{2\varepsilon/3}}}[\big]$.
\end{restatable}

\Cref{thm:SIRStight,thm:cSIRSlower} show that the immunity in each process has drastically different effects, although the expected degree of immunity is the same in either process.
In the SIRS process, the full immunity guarantees that the expected survival time does not become super-polynomial.
However, in the \cSIRS process, the expected survival time is very similar to that of the SIS process (see also \Cref{table:results}), the latter of which does not exhibit any immunity at all.
Hence, immunity does not seem to be very useful if it cannot be guaranteed at full levels for a certain amount of time.

Our mathematical analysis for the SIRS and the \cSIRS process is very similar, as the processes are defined rather similarly. Thus, we first prove useful statements that hold for both processes. We note that instead of analyzing the \cSIRS process directly, we analyze the \placy process in order to use its shared notation with the SIRS process.

Our general proof strategy is to first show that it is very unlikely that there are ever too many recovered leaves. In turn, while there are not many infected vertices, there are almost always enough susceptible vertices to infect. Since infections cannot spread once the center of the star is not infected, we split the processes into \emph{center-healthy} phases and \emph{center-infected} phases. We show that center-infected phases have a constant probability to end with at least $\infectionRate d \numberOfVertices$ infected vertices for some constant $d \in \R_{>0}$. We then show that the process needs a lot of center-healthy phases in order to heal these $\infectionRate d \numberOfVertices$ leaves. As each center-infected phase in between has a high enough probability to get back up to these many infected leaves, the process survives relatively long until then. However, the actual details and results differ a lot between the two processes in this last step.

Note that our lower bounds for the survival time also hold when the star is only a subgraph of the underlying graph and higher infection rates lead to stronger lower bounds. This normally does not have to be the case in the SIRS or \cSIRS process as additional infections could lead to more recovered vertices that block the infection.

\begin{restatable}{observation}{monotone}\label{obs:monotone}
  Let $G$ be a star with $\numberOfVertices$ leaves. Let \contactProcess be a SIRS process or a \placy process on $G$ with infection rate~$\infectionRate$. All of our lower bounds for the expected survival time of \contactProcess also hold when the process runs on a supergraph of $G$ (with same parameters and starting configuration on vertices of $G$). The lower bounds also hold for processes on the same graph with infection rate~$\infectionRate' \in \R_{>\infectionRate}$. In particular, our lower bounds for $\infectionRate =1$ also hold for all infection rates $\infectionRate'>1$.
\end{restatable}

\begin{proof}
  Our results rely heavily on the fact that there are likely never too many recovered leaves (\Cref{lem:recoveredBound}). This statement even holds when susceptible vertices immediately become infected again. Thus, in particular, it holds for any supergraph of the star as well and for all infection rates $\infectionRate'$. In the rest of the proofs, we only use lower bounds on the rate at which we infect new vertices. As adding more vertices and edges to the graph and increasing the infection rate only increases the rate at which new vertices get infected, the used lower bounds on the rates still hold for supergraphs and larger infection rates. Hence the proofs still go through exactly the same way.
\end{proof}

\subsection{Center-Infected Phase}

We show that with constant probability, a center-infected phase on the star ends with at least $d \infectionRate \numberOfVertices$ infected vertices, for some constant $d \in \R_{>0}$. This holds for both the SIRS and the \placy process. To get this bound, we first show that both processes likely never reach a state with too many recovered vertices.

\begin{restatable}{lemma}{recoveredBound}
  \label{lem:recoveredBound}
  Let $G$ be a star with $\numberOfVertices$ leaves. Let \contactProcess be a SIRS process or a \placy process on $G$ with infection rate~$\infectionRate$ and with constant deimmunization rate \deimmunizationRate (or constant resistance decay rate $\decayRate$ respectively, but we refer to it as \deimmunizationRate until the end of the statement). Let $\timeDiscrete \in \N$, and let $\recoveredDiscrete{\timeDiscrete} \leq \frac{2}{2+\deimmunizationRate}\numberOfVertices + 1$. Then the probability $p$ that $\recoveredDiscrete{\timeDiscrete}$ reaches $\frac{2+\deimmunizationRate/2}{2+\deimmunizationRate}\numberOfVertices$ before reaching $\frac{2}{2+\deimmunizationRate}\numberOfVertices$ is at most $(2^{\frac{\deimmunizationRate}{2(2+\deimmunizationRate)}\numberOfVertices}-1)^{-1}$.
\end{restatable}

\begin{proof}
  Consider the process $(X_t)_{t \in \N}$ that decreases by one in each step with probability $\frac{\deimmunizationRate X_\timeDiscrete}{\deimmunizationRate X_\timeDiscrete + (\numberOfVertices-X_\timeDiscrete)}$ and increases by one with probability $\frac{\numberOfVertices - X_\timeDiscrete}{\deimmunizationRate X_\timeDiscrete + (\numberOfVertices-X_\timeDiscrete)}$. Note that $(X_\timeDiscrete)$ processes the number of recovered vertices at time~$\tau_t$ of the process in which each recovered vertex immediately goes into the infected state again instead of the susceptible one. Thus, $X_\timeDiscrete$ dominates $\recoveredDiscrete{\timeDiscrete}$.

  As long as $X_\timeDiscrete \geq \frac{2}{2+\deimmunizationRate}\numberOfVertices$, it decreases by one with a probability of at least $\frac{\deimmunizationRate X_\timeDiscrete}{\deimmunizationRate X_\timeDiscrete + (\numberOfVertices-X_\timeDiscrete)} \geq \frac{\frac{2 \deimmunizationRate}{2+\deimmunizationRate}\numberOfVertices}{\frac{2 \deimmunizationRate}{2+\deimmunizationRate}\numberOfVertices + \frac{\deimmunizationRate}{2+\deimmunizationRate}\numberOfVertices} \geq \frac{2}{3}$. It is therefore dominated by a gambler's ruin with a biased coin of probability $1/3$ of winning. Applying the gambler's ruin results from \Cref{pre:gamblersRuin} concludes the proof.
  \begin{align*}
    p & \leq \frac{1-\left(\frac{2/3}{1/3}\right)^{(\frac{2}{2+\deimmunizationRate}\numberOfVertices + 1) - (\frac{2}{2+\deimmunizationRate}\numberOfVertices)}}{1-\left(\frac{2/3}{1/3}\right)^{(\frac{2+\deimmunizationRate/2}{2+\deimmunizationRate}\numberOfVertices) - (\frac{2}{2+\deimmunizationRate}\numberOfVertices)}} \\
      & = \frac{1-2^{1}}{1-2^{\frac{\deimmunizationRate}{2(2+\deimmunizationRate)}\numberOfVertices}}                                                                                                                                                                                                                            \\
      & = \frac{1}{2^{\frac{\deimmunizationRate}{2(2+\deimmunizationRate)}\numberOfVertices}-1}.\qedhere
  \end{align*}
\end{proof}

Although \Cref{lem:recoveredBound} only shows an exponentially low probability of reaching too many recovered vertices during a single phase, in the following, we often assume that the process \emph{never} does so before it dies out.
This assumption makes sense, as we only make it for statements for which we show a sub-exponential expected survival time.
Thus, due to the exponentially small probability in \Cref{lem:recoveredBound}, the probability of ever having too many recovered vertices before the process dies out is overall still sub-constant.
Hence, conditioning on this never occurring does not change our arguments in the following asymptotically.
Moreover, once the process dies out under the previous event, it can never reach too many recovered vertices anymore.

Let $c= \frac{\deimmunizationRate}{4(2+\deimmunizationRate)}$. \Cref{lem:recoveredBound} shows that the process is exponentially unlikely to reach a state without at least $2 c \numberOfVertices$ vertices that are not recovered. Hence, as long as there are at most $c \numberOfVertices$ infected vertices, there are very likely at least $c \numberOfVertices$ susceptible vertices. We use this fact to first show that, with constant probability, we reach a state with sufficiently many infected vertices after a center-infected phase.

\begin{restatable}{lemma}{endProbability}
  \label{lem:endProbability}
  Let $G$ be a star with $\numberOfVertices$ leaves. Let \contactProcess be a SIRS process or a \placy process on $G$ with infection rate~$\infectionRate\leq 1$ with $\infectionRate \in \omega(n^{-1})$ and with constant deimmunization rate \deimmunizationRate (or constant resistance decay rate $\decayRate$ respectively, but we refer to it as \deimmunizationRate until the end of the statement). Let $\varepsilon \in \R_{>0}$ be a constant and let $\timeDiscrete \in \N$ such that the center is infected at time $\timeContinuous{\timeDiscrete}$. Furthermore, let $c= \frac{\deimmunizationRate}{4(2+\deimmunizationRate)}$. Assume that there are always at least $2 c \numberOfVertices$ vertices that are not recovered during the considered time interval. Let $d \in \R_{>0}$ be a constant with $d \leq c/7$ and $\eulerE^{-2d/c}\geq 1- \varepsilon/2$. Then starting from $\timeContinuous{\timeDiscrete}$, the probability of the event $E$ that we reach a state with at least $\infectionRate d \numberOfVertices$ infected vertices before the center heals is at least $1- \varepsilon$ for sufficiently large $\numberOfVertices$.
\end{restatable}

\begin{proof}
  The constant $d$ is chosen small enough such that there are with high probability at least $2 \infectionRate d \numberOfVertices$ relevant triggers before the center heals and that with high probability at least $3/4$ of them infect vertices, which is enough to reach $\infectionRate d \numberOfVertices$ infected vertices.

  Let $E_0$ be the event that there are at least $2 \infectionRate d \numberOfVertices$ triggers that heal or infect leaves before the center heals. We first calculate $\Pr{E_0}$. The center heals at a rate of $1$. We assume that there are always at least $2 c \numberOfVertices$ leaves that are not recovered. While the center is infected, these vertices change state with rate at $\infectionRate$ or $1$ depending on whether they are susceptible or infected. As $\infectionRate \leq 1$, leaves heal or infect at a rate of at least $2 \infectionRate c \numberOfVertices$. Therefore, the number of leaves that heal or get infected before the center heals dominate a geometric variable $A$ with parameter $\frac{1}{2 \infectionRate c \numberOfVertices}$. We can therefore lower bound $\Pr{E_0}$ by

  \begin{align*}
    \Pr{E_0} & \geq \Pr{A \geq 2 \infectionRate d \numberOfVertices}                                              \\
             & = (1- \frac{1}{2 \infectionRate c \numberOfVertices})^{2 \infectionRate d \numberOfVertices}       \\
             & \geq \eulerE^{-2\frac{2 \infectionRate d \numberOfVertices}{2 \infectionRate c \numberOfVertices}} \\
             & = \eulerE^{-2\frac{d}{c}}                                                                          \\
             & \geq 1- \varepsilon/2.
  \end{align*}

  Now we lower bound $\Pr{E}[E_0]$. Note that $\infectionRate d \numberOfVertices \leq c \numberOfVertices$, so as long as there are less than $\infectionRate d \numberOfVertices$ infected vertices, there are at least $c \numberOfVertices$ susceptible vertices, so leaves get infected at a rate of at least $\infectionRate c \numberOfVertices$. As $d \leq c/7$, leaves heal at a rate of at most $\infectionRate d \numberOfVertices \leq \infectionRate c \numberOfVertices/7$. So while the process is below $\infectionRate d \numberOfVertices$ infected vertices, every trigger that heals or infects leaves has a probability of at least $7/8$ to be an infection trigger. Therefore conditioned on $E_0$, the number of vertices that get infected in the first $2 \infectionRate d \numberOfVertices$ triggers before the center heals dominates a Binomial random variable $B \sim \text{Bin}(2 \infectionRate d \numberOfVertices,7/8)$. If at least $3/4$ of those triggers are infection triggers, than at most $1/4$ of them can heal leaves, which means that at least $\infectionRate d \numberOfVertices$ new vertices get infected. That implies that $E$ happened. We can now lower bound $\Pr{E}[E_0]$ using Chernoff bounds.

  \begin{align*}
    \Pr{E}[E_0] & \geq \Pr{B \geq 3/4 \cdot 2 \infectionRate d \numberOfVertices}        \\
                & = \Pr{B \geq 6/7 \cdot 7/8 \cdot 2 \infectionRate d \numberOfVertices} \\
                & = \Pr{B \geq (1-\frac{1}{7})\E{B}}                                     \\
                & \geq 1-\eulerE^{-(1/7)^2 \E{B}/2}                                      \\
                & \geq 1-\varepsilon/2.
  \end{align*}

  The last step holds for sufficiently large $\numberOfVertices$ as $\infectionRate \in \omega(n^{-1})$.

  We now bound $\Pr{E}$ using the two previous bounds and the law of total probability. We get

  \begin{align*}
    \Pr{E} & \geq \Pr{E}[E_0] \cdot \Pr{E_0}                   \\
           & \geq (1- \varepsilon/2) \cdot (1 - \varepsilon/2) \\
           & \geq 1- \varepsilon.\qedhere
  \end{align*}
\end{proof}

For the rest of the analysis of the survival time, the two processes differ.
Hence, we analyze them separately.

\subsection{The SIRS Process}

For the SIRS process, the idea of the proof is as follows. We consider the number of infected leaves while the center is recovered. When the center becomes susceptible and loses its immunity, we condense the following center-susceptible and center-infected phase into one step. We analyze the number of infected leaves of the resulting process between some lower bound~$\ell$ and some upper bound~$u$. We upper-bound the probability of dropping down to $\ell$ from $u-1$ before reaching $u$ again. This gives us a lower bound on how many of these phases happen in expectation. As each of these phases includes the center losing immunity, they have an expected constant length.

The process we consider is quite simple, as it just decreases the number of infected vertices by $1$ at rate $\infectedDiscrete{\timeDiscrete}$ and starts a center-susceptible and center-infected phase at rate $\deimmunizationRate$. For the latter, we showed in \Cref{lem:endProbability} that it reaches a state with at least $u$ infected vertices with constant probability and show here that it is very unlikely to ever reduce the number of infected vertices by more than $\ell$. Essentially, we use~$\ell$ vertices as a buffer that is used for phases where the center is not recovered. Thus, we consider vertices that heal in these phases to not decrease $\infectedDiscrete{\timeDiscrete}$. However, as long as these phases do not heal~$\ell$ vertices and the process does not fall below $\ell$, we know that the original process cannot have died out yet.

\Cref{lem:endProbability} shows that each center-infected phase has a constant probability of infecting more than $\infectionRate d \numberOfVertices$ vertices, for some constant $d \in \R_{>0}$. We now show that the center-infected phases that do not achieve this together with their preceding center-susceptible phases have a very low probability of healing too many vertices.

\begin{restatable}{lemma}{susInfPhase}
  \label{lem:susInfPhase}
  Let $G$ be a star with $\numberOfVertices$ leaves. Let \contactProcess be a SIRS process on $G$ with infection rate~$\infectionRate\leq 1$ and with constant deimmunization rate \deimmunizationRate. Let $\varepsilon \in \R_{>0}$ be a constant and let $\timeDiscrete \in \N$ such that the center is susceptible at time $\timeContinuous{\timeDiscrete}$. For $c= \frac{\deimmunizationRate}{4(2+\deimmunizationRate)}$, assume that there are always at least $2 c \numberOfVertices$ vertices that are not recovered during the considered time interval. Let $d \in \R_{>0}$ be a constant with $d \leq c/7$. Then starting from~$\timeContinuous{\timeDiscrete}$, the probability of the event $E$ that we reach a state with at most $\infectedDiscrete{\timeDiscrete} - 2\infectionRate^{-1}n^{\varepsilon}$ infected vertices before either the center recovers or we reach at least $\infectionRate d \numberOfVertices$ infected vertices is at most $2 \eulerE^{-n^{\varepsilon}/2}$ for sufficiently large $\numberOfVertices$.
\end{restatable}

\begin{proof}
  The idea of the proof is to consider the center-susceptible phase and the center-infected phases separately and show that both of them are very unlikely to heal $\infectionRate^{-1}n^{\varepsilon}$ leaves. Let $E_0$ be the event that in the center-susceptible phase $\infectionRate^{-1}n^{\varepsilon}$ leaves are healed. Furthermore, let $E_1$ be the event that in the center-infected phase $\infectionRate^{-1}n^{\varepsilon}$ more leaves heal than new leaves being infected before either the center heals or the number of infected vertices exceeds $\infectionRate d \numberOfVertices$.

  Consider the center-susceptible phase. In a state with $I$ infected leaves, they heal at a rate of $I$ and they infect the center at a rate of $\infectionRate I$. Therefore, the number of leaves that heal before the center gets infected follows a geometrical distribution $A$ with parameter $\frac{\infectionRate}{1+\infectionRate}$ that is capped at the number of initially infected leaves. We get

  \begin{align*}
    \Pr{E_0} & \leq \Pr{A \geq \infectionRate^{-1}n^{\varepsilon}}                                \\
             & = (1-\frac{\infectionRate}{1+\infectionRate})^{\infectionRate^{-1}n^{\varepsilon}} \\
             & \leq (1-\frac{\infectionRate}{2})^{\infectionRate^{-1}n^{\varepsilon}}             \\
             & \leq \eulerE^{-n^{\varepsilon}/2}.
  \end{align*}

  Now consider the center-infected phase. As we assume there to be at least $2c\numberOfVertices$ vertices that are not recovered during that phase and because $d \leq c/7$ and $\infectionRate \leq 1$, as long there are less than $\infectionRate d \numberOfVertices$ infected vertices, there are at least $c \numberOfVertices$ susceptible vertices. Hence, vertices heal at a rate of at most $\infectionRate d \numberOfVertices$ and new leaves get infected at a rate of at least $\infectionRate c \numberOfVertices$. Therefore the probability in each step to infect a new leaf is at least 7 times as high as the probability to heal one. That means that the number of infected leaves is dominated by a gambler's ruin instance with a biased coin with probability $7/8$. Let $p$ be the probability that starting at 0, this gambler's ruin instance drops to $- \infectionRate^{-1}n^{\varepsilon}$ before reaching $\infectionRate d \numberOfVertices$. We get

  \begin{align*}
    \Pr{E_1} & \leq p                                                                                                                              \\
             & \leq \frac{7^{\infectionRate d \numberOfVertices}-1}{7^{\infectionRate d \numberOfVertices +\infectionRate^{-1}n^{\varepsilon} }-1} \\
             & \leq \frac{7^{\infectionRate d \numberOfVertices}}{7^{\infectionRate d \numberOfVertices +\infectionRate^{-1}n^{\varepsilon}-1 }}   \\
             & = 7^{-\infectionRate^{-1}n^{\varepsilon}+1 }                                                                                        \\
             & \leq \eulerE^{-n^{\varepsilon}/2}.
  \end{align*}

  The last inequality holds for sufficiently large $\numberOfVertices$ as $\infectionRate \leq 1 $. Now because we defined our events such that $E$ implies $E_0$ or $E_1$, we get

  \begin{align*}
    \Pr{E} & \leq \Pr{E_0 \lor E_1}                       \\
           & \leq \Pr{E_0}+\Pr{E_1}                       \\
           & \leq 2 \eulerE^{-n^{\varepsilon}/2}.\qedhere
  \end{align*}
\end{proof}

\Cref{lem:endProbability} shows that each center-infected phase has a constant probability of infecting more than $\infectionRate d \numberOfVertices$ vertices, and \Cref{lem:susInfPhase} shows that each center-infected phase and center-susceptible phase does not heal too many vertices. We now combine these two results to show that all center-infected phases and center-susceptible phases together do not heal too many vertices before they infect more than $\infectionRate d \numberOfVertices$ vertices.

\begin{restatable}{corollary}{corSusInfPhase}
  \label{cor:susInfPhase}
  Let $G$ be a star with $\numberOfVertices$ leaves. Let \contactProcess be a SIRS process on $G$ with infection rate~$\infectionRate\leq 1$ with $\infectionRate \in \omega(n^{-1})$ and with constant deimmunization rate \deimmunizationRate. Let $\varepsilon_0, \varepsilon_1 \in \R_{>0}$ be constants with $\varepsilon_0 \leq \eulerE^{-1}$ and let $\timeDiscrete \in \N$ such that the center is susceptible at time $\timeContinuous{\timeDiscrete}$. For $c= \frac{\deimmunizationRate}{4(2+\deimmunizationRate)}$, assume that there are always at least $2 c \numberOfVertices$ vertices that are not recovered during the considered time interval. Let $d \in \R_{>0}$ be a constant with $d \leq c/7$ and $\eulerE^{-2d/c}\geq 1- \varepsilon_0/2$. We define the random process $(X_{t'})_{t' \in \N_{\geq t}}$ to be $0$ at step $t$, to increase by $1$ in every step in which a leaf is healed in a center-susceptible or center-infected phase and to decrease by one in every step it is positive and a leaf gets infected. Then starting from step~$\timeDiscrete$, the probability of the event $E$ that we reach a time step $t'$ with $X_{t'} \geq 2\infectionRate^{-1}n^{\varepsilon_1}$ before we reach a time step with at least $\infectionRate d \numberOfVertices$ infected vertices is at most $\eulerE^{-\numberOfVertices^{\varepsilon_1/4}}$.
\end{restatable}

\begin{proof}
  The idea is to first use \Cref{lem:endProbability} to bound the number of complete phases (time period between two times the center becomes susceptible) before reaching $\infectionRate d \numberOfVertices$ infected vertices and then using \Cref{lem:susInfPhase} and a union bound to bound the probability that any of those phases heals enough leaves. Let $E_0$ be the event that there are at least $\numberOfVertices^{\varepsilon_1/2}$ center-susceptible phases before we reach a state with $\infectionRate d \numberOfVertices$ infected vertices. Let $E_1$ be the event that there is a center-susceptible phase with following center-infected phase that increase $X$ by at least $2\infectionRate^{-1}n^{\varepsilon_1/2}$ before we reach a state with $\infectionRate d \numberOfVertices$ infected vertices.

  Let $Y$ be the random variable that counts the number of center-susceptible phases before the process reaches a state with $\infectionRate d \numberOfVertices$ infected vertices. By \Cref{lem:endProbability}, each center-infected phase has a probability of at least $1- \varepsilon_0$ to reach $\infectionRate d \numberOfVertices$ vertices. Hence, $Y$ is dominated by a geometric random variable $A$ with parameter $1-\varepsilon_0$. We get

  \begin{align*}
    \Pr{E_0} & = \Pr{Y \geq \numberOfVertices^{\varepsilon_1/2}}      \\
             & \leq \Pr{A \geq \numberOfVertices^{\varepsilon_1/2}}   \\
             & = \varepsilon_0^{\numberOfVertices^{\varepsilon_1/2}}.
  \end{align*}

  By \Cref{lem:susInfPhase}, each center-susceptible phase and the following center-infected phase have a probability of at most $2\eulerE^{-\numberOfVertices^{\varepsilon_1/2}/2}$ to increase $X_{t'}$ by more than $2 \infectionRate^{-1}\numberOfVertices^{\varepsilon_1/2}$. Conditioning on $\overline{E_0}$ and using a Union bound gives us

  \begin{align*}
    \Pr{E_1}[\overline{E_0}] & \leq \numberOfVertices^{\varepsilon_1/2} \cdot 2\eulerE^{-\numberOfVertices^{\varepsilon_1/2}/2}.
  \end{align*}

  By pigeon-hole principle, conditioned on $\overline{E_0}$, $E$ implies $E_1$. Together with the law of total probability that gives us

  \begin{align*}
    \Pr{E} & \leq \Pr{E}[E_0]\cdot \Pr{E_0} + \Pr{E}[\overline{E_0}]\cdot \Pr{\overline{E_0}}                                                                       \\
           & \leq \Pr{E_0} + \Pr{E}[\overline{E_0}]                                                                                                                 \\
           & \leq \Pr{E_0} + \Pr{E_1}[\overline{E_0}]                                                                                                               \\
           & \leq \varepsilon_0^{\numberOfVertices^{\varepsilon_1/2}} + \numberOfVertices^{\varepsilon_1/2} \cdot 2\eulerE^{-\numberOfVertices^{\varepsilon_1/2}/2} \\
           & \leq \eulerE^{-\numberOfVertices^{\varepsilon_1/2}} + \numberOfVertices^{\varepsilon_1/2} \cdot 2\eulerE^{-\numberOfVertices^{\varepsilon_1/2}/2}      \\
           & \leq \eulerE^{-\numberOfVertices^{\varepsilon_1/4}}.
  \end{align*}

  The second to last step holds for $\varepsilon_0 \leq \eulerE^{-1}$ and the last step holds for sufficiently large $\numberOfVertices$.
\end{proof}

\Cref{cor:susInfPhase} shows that center-susceptible phases and center-infected phases do not heal that many leaves in total. Thus, in order to bound the probability of the infection dying out, we mainly consider the center-recovered phases. We capture this in the following lemma.

\begin{restatable}{lemma}{lemSIRSlower}
  \label{lem:SIRSlower}
  Let $G$ be a star with $\numberOfVertices$ leaves. Let \contactProcess be a SIRS process on $G$ with infection rate~$\infectionRate\leq 1$ with $\infectionRate \in \omega(n^{-1})$ and with constant deimmunization rate \deimmunizationRate. Let $\varepsilon_0, \varepsilon_1 \in \R_{>0}$ be constants with $\varepsilon_0\leq \eulerE^{-1}$. For $c= \frac{\deimmunizationRate}{4(2+\deimmunizationRate)}$, assume that there are always at least $2 c \numberOfVertices$ vertices that are not recovered during the considered time interval. Let $d \in \R_{>0}$ be a constant with $d \leq c/7$ and $\eulerE^{-2d/c}\geq 1- \varepsilon_0/2$ and let $\timeDiscrete \in \N$ such that there are $\infectionRate d \numberOfVertices-1$ infected vertices at time $\timeContinuous{\timeDiscrete}$. Then starting from $\timeContinuous{\timeDiscrete}$, the probability of the event $E$ that the infection dies out before it reaches at least $\infectionRate d \numberOfVertices$ infected vertices is at most $\bigTheta{\infectionRate^{-2(1-\varepsilon_0)\deimmunizationRate}\numberOfVertices^{-(1-\varepsilon_0)(1-\varepsilon_1)\deimmunizationRate}}$.
\end{restatable}

\begin{proof}
  The idea is to merge each center-susceptible phase and the following center-infected phase into a single step of the process. We then use \Cref{lem:endProbability} to argue that each of those steps infects $\infectionRate d \numberOfVertices$ vertices with probability at least $1-\varepsilon_0$. Additionally, by \Cref{cor:susInfPhase}, those steps heal at most $2\infectionRate^{-1}n^{\varepsilon_1}$ leaves. Therefore, for the infection to die out, the other $\infectionRate d \numberOfVertices -2\infectionRate^{-1}n^{\varepsilon_1}$ leaves have to heal in center-recovered phases before one of the center-infected phases infects $\infectionRate d \numberOfVertices$ new vertices.

  Let $X$ be the number of vertices that heal during center infected and center-susceptible phases minus the number of newly infected leaves in those phases. By \Cref{cor:susInfPhase}, $\Pr{X \geq 2\infectionRate^{-1}\numberOfVertices^{\varepsilon_1}} \leq \eulerE^{-\numberOfVertices^{\varepsilon_1/4}}$. We now condition on the event $E_1$ that $X < 2\infectionRate^{-1}\numberOfVertices^{\varepsilon_1}$. Consider a state in which already $Y$ vertices healed in center-recovered phases and in which the center is recovered. Then there are at most $\infectionRate d \numberOfVertices - Y$ infected leaves that each heal with a rate of $1$. The center deimmunizes at a rate of $\deimmunizationRate$. By \Cref{lem:endProbability}, we reach $\infectionRate d \numberOfVertices$ infected vertices before the center recovers again with a probability of at least $1-\varepsilon_0$. We ignore the center susceptible and center-infected phases that do not achieve that as we already accounted for them in $X$. The event $E$ happens only when $\infectionRate d \numberOfVertices - 2\infectionRate^{-1}\numberOfVertices^{\varepsilon_1}$ leaves heal in center-recovered phases before a center-infected phase infects $\infectionRate d \numberOfVertices$ leaves, which happens at a rate of $(1-\varepsilon_0)\deimmunizationRate$. As we aim to upper bound the probability of $E$ happening, we assume the worst case of all the center healing rates, i.e. we assume that our buffer is only used up after healing all the vertices in center-recovered phases. Together with \Cref{pre:gamma}, that gives us

  \begin{align*}
    \Pr{E}[E_1] & \leq \prod_{i=2\infectionRate^{-1}\numberOfVertices^{\varepsilon_1}}^{\infectionRate d \numberOfVertices}{\frac{i}{i+(1-\varepsilon)\deimmunizationRate}}                                         \\
                & \in \bigTheta{\frac{(2\infectionRate^{-1}\numberOfVertices^{\varepsilon_1})^{(1-\varepsilon_0)\deimmunizationRate}}{(\infectionRate d \numberOfVertices)^{(1-\varepsilon_0)\deimmunizationRate}}} \\
                & \in \bigTheta{\infectionRate^{-2(1-\varepsilon_0)\deimmunizationRate}\numberOfVertices^{-(1-\varepsilon_0)(1-\varepsilon_1)\deimmunizationRate}}.
  \end{align*}

  Getting rid of the conditioning on $E_1$ we get

  \begin{align*}
    \Pr{E} & \leq \Pr{\overline{E_1}} + \Pr{E}[E_1]                                                                                                                                                           \\
           & \leq \eulerE^{-\numberOfVertices^{\varepsilon_1/4}} + \bigTheta{\infectionRate^{2(1-\varepsilon_0)\deimmunizationRate}\numberOfVertices^{(1-\varepsilon_0)(1-\varepsilon_1)\deimmunizationRate}} \\
           & = \bigTheta{\infectionRate^{2(1-\varepsilon_0)\deimmunizationRate}\numberOfVertices^{(1-\varepsilon_0)(1-\varepsilon_1)\deimmunizationRate}}.\qedhere
  \end{align*}
\end{proof}

We now combine these results into a lower bound for the expected survival time.

\begin{restatable}{theorem}{thmSIRSlower}
  \label{thm:SIRSlower}
  Let $G$ be a star with $\numberOfVertices$ leaves. Let \contactProcess be a SIRS process on $G$ with infection rate~$\infectionRate\leq 1$ with $\infectionRate \in \omega(n^{-1})$ and with constant deimmunization rate \deimmunizationRate that starts with an infected center and no recovered leaves. Let $\varepsilon \in \R_{>0}$ be a constant, and let $T$ be the survival time of $C$. Then $\E{T}  \in \bigOmega{(\infectionRate^2\numberOfVertices)^{(1-\varepsilon)\deimmunizationRate}}$.
\end{restatable}

\begin{proof}
  Let $c= \frac{\deimmunizationRate}{4(2+\deimmunizationRate)}$. We showed in \Cref{lem:recoveredBound}, that it is very unlikely that there are ever more than $(1-2c)n$ recovered vertices. We assume in the rest of the proof that there are never more than $(1-2c)n$ recovered vertices, which is true with probability at least $1/2$ for sufficiently large $n$. By law of total probability, the overall expected survival time is then at least half of the expected survival time conditioned on our assumption.

  Let $\varepsilon_0, \varepsilon_1 \in \R_{>0}$ be constants with $\varepsilon_0\leq \eulerE^{-1}$ and $(1-\varepsilon_0)(1-\varepsilon_1)\geq 1-\varepsilon$. Let $d \in \R_{>0}$ be a constant with $d \leq c/7$ and $\eulerE^{-2d/c}\geq 1- \varepsilon_0/2$. We assume that the infection reaches a state with $\infectionRate d \numberOfVertices$ infected vertices. By \Cref{lem:endProbability}, this happens with constant probability, so again this assumption only changes our expected survival time by at most a constant factor.

  Let $X$ be the number of times that the process reaches a state with $\infectionRate d \numberOfVertices$ infected vertices after dropping below that number of infected vertices. By \Cref{lem:SIRSlower}, $X$ dominates a geometric random variable with a parameter in $\bigTheta{\infectionRate^{-2(1-\varepsilon_0)\deimmunizationRate}\numberOfVertices^{-(1-\varepsilon_0)(1-\varepsilon_1)\deimmunizationRate}}$. Hence $\E{X}\in \bigOmega{\infectionRate^{2(1-\varepsilon)\deimmunizationRate}\numberOfVertices^{(1-\varepsilon)\deimmunizationRate}}$. As most phases of $X$ include the center being infected and the expected time between two center infections is at least constant as the center has to heal in between, we get $\E{T}\in \bigOmega{(\infectionRate^2\numberOfVertices)^{(1-\varepsilon)\deimmunizationRate}}$.
\end{proof}

We now show that this lower bound is actually tight when disregarding sub-polynomial factors.

\begin{restatable}{theorem}{thmSIRSupper}
  \label{thm:SIRSupper}
  Let $G$ be a star with $\numberOfVertices$ leaves. Let \contactProcess be a SIRS process on $G$ with infection rate~$\infectionRate\leq 1$ with $\infectionRate \in \Omega(n^{-1/2})$ and with constant deimmunization rate \deimmunizationRate that starts with infected center and no recovered leaves. Let $T$ be the survival time of $C$. Then $\E{T}  \in \bigO{(\infectionRate^2\numberOfVertices)^{\deimmunizationRate}+\log \numberOfVertices}$.
\end{restatable}

\begin{proof}
  First note that the theorem statement is already known for $\infectionRate \in \bigTheta{n^{-1/2}}$ as the statement is then that $\E{T} \in \bigO{\log n}$, which is already known \cite{ganesh2005effect}. Hence, for the rest of the proof we can assume $\infectionRate \geq n^{-1/2}$.

  The idea is to split the process into center-infected phases, center-recovered phases and center-susceptible phases. For the center-infected phases we show that they very likely do not end with more than $3\infectionRate \numberOfVertices$ infected leaves. For the center-recovered phase we show that it has a sufficiently large probability to drop down to $\infectionRate^{-1}$ infected leaves. The center-susceptible phase then has a constant probability to let the infection die out when it starts with at most $\infectionRate^{-1}$ infected leaves. As those three phases together take in expectation at most constant time, we get the desired bound.

  Consider a state of the infection with at least $2 \infectionRate \numberOfVertices$ infected leaves. Then leaves heal at a rate of at least $2 \infectionRate \numberOfVertices$ while new leaves get infected at a rate of at most $\infectionRate \numberOfVertices$. Hence the number of infected leaves is dominated by a biased gambler's ruin instance with a probability of $1/3$ to increase. Therefore, the number of infected vertices is exponentially (in $\infectionRate \numberOfVertices$) unlikely to go above $3 \infectionRate \numberOfVertices$ and when it does, it drops back down to at most $2 \infectionRate \numberOfVertices$ infected leaves quickly. Each time the center recovers, it is therefore very likely that there are at most $3 \infectionRate \numberOfVertices$ infected leaves.

  Now consider a center-recovered phase that starts with at most $3 \infectionRate \numberOfVertices$ infected leaves. We aim to calculate the probability $p$ of the event that there are at most $\infectionRate^{-1}$ infected leaves the moment the center loses its immunity. For this, the healing clocks have to trigger before the center deimmunization clock often enough. Hence, using \Cref{pre:gamma}, for $\infectionRate^{-1} \leq 3 \infectionRate \numberOfVertices$, $p$ is lower bounded by

  \begin{align*}
    p & \geq \prod_{i=\infectionRate^{-1}}^{3 \infectionRate \numberOfVertices}{\frac{i}{i+\deimmunizationRate}}                       \\
      & \in \bigTheta{\frac{(\infectionRate \numberOfVertices)^{-\deimmunizationRate}}{\infectionRate^{-1\cdot -\deimmunizationRate}}} \\
      & \in \bigTheta{(\infectionRate^2\numberOfVertices)^{-\deimmunizationRate}}.
  \end{align*}

  At last, consider a center-susceptible phase that starts with at most $\infectionRate^{-1}$ infected leaves. Each infected leaf heals at a rate of 1 and infects the center at a rate of $\infectionRate$. That means that the number of leaves that heal before the center gets infected follows a geometric distribution with parameter $\frac{\infectionRate}{1+\infectionRate}$ (that is capped at the initial number of infected vertices). So the probability $q$ of the infection dying out before infecting the center is at least

  \begin{align*}
    q & \geq \prod_{i=1}^{\infectionRate^{-1}}{\frac{1}{1+\infectionRate}}            \\
      & = \left(\frac{1}{1+\infectionRate}\right)^{1/\infectionRate}                  \\
      & = \left(1 - \frac{\infectionRate}{1+\infectionRate}\right)^{1/\infectionRate} \\
      & \geq (1 - \infectionRate)^{1/\infectionRate}                                  \\
      & \geq \eulerE^{-2}.
  \end{align*}

  Putting all these phases together, we get that each time the center heals, there is a probability of $\bigOmega{p \cdot q} = \bigOmega{(\infectionRate^2\numberOfVertices)^{-\deimmunizationRate}}$ for the infection to die out before the center gets infected again. Hence, the number of times the center gets infected is dominated by a geometric random variable with this parameter. We get that the center is infected in expectation at most $\bigO{(\infectionRate^2\numberOfVertices)^{\deimmunizationRate}}$ times.

  We now bound the time between two center infections. Starting with infected center, in order for the center to get infected again, it has to recover, lose immunity and then get infected again. Both the recovering and the losing of immunity take in expectation constant time. The infection might take longer. However, as long as there are at least $\infectionRate^{-1}$ infected leaves, the center gets infected at rate at least 1, so it takes an expected constant time to be infected. When there are less than $\infectionRate^{-1}$ infected leaves, we showed that the infection dies out with constant probability before the center is infected. So this happens in expectation only a constant number of times. These phases take at most as much time as it takes to trigger the recovering clocks of all leaves at least once. As there are in expectation $n$ recovery triggers in a time interval of length 1, by Coupon collector results, triggering all recovery clocks once takes in expectation $\log(\numberOfVertices)$ time. Therefore, all the phases we considered take expected constant time except for constantly many that take at most $\log(\numberOfVertices)$ time. As there are at most $\bigO{(\infectionRate^2\numberOfVertices)^{\deimmunizationRate}}$ of those phases in expectation before the infection dies out, it follows $\E{T}  \in \bigO{(\infectionRate^2\numberOfVertices)^{\deimmunizationRate}+\log \numberOfVertices}$.
\end{proof}

\subsection{The \cSIRS Process}

The idea of the proof for the \cSIRS process is to show that the process stays at a number of infected leaves for a long time, in which it likely needs a super-polynomial number of center-healthy phases for the infection to die out. In order to show this, we bound the number of leaves that heal in a center-healthy phase. To this end, we split the center-healthy phase into two intervals, in which the first one is just there to pass some time such that the resistance in the second interval is small enough such that the center very likely gets infected during this interval. Recall that we consider the \placy process instead of the \cSIRS process, as this lets us use some results from the SIRS process.

We start with upper-bounding the number of leaves that heal in a time interval of length $t$.

\begin{restatable}{lemma}{healingLeaves}
  \label{lem:healingLeaves}
  Let $G$ be a star with $\numberOfVertices$ leaves. Let \contactProcess be a \placy process on $G$ with infection rate $\infectionRate \in \omega(n^{-1})$ and with constant resistance decay rate \decayRate. Let $d \in \R_{>0}$ be a constant and let $t' \in \R_{\geq 0}$ such that there are at most $d \infectionRate \numberOfVertices$ infected leaves at time $t'$. Furthermore, let $t \in [0,1]$. Then the probability that more than $2td \infectionRate \numberOfVertices$ of the infected leaves heal until time $t' +t$ is at most $\eulerE^{-\frac{t d\infectionRate \numberOfVertices}{6}}$.
\end{restatable}

\begin{proof}
  Each of the infected leaves has a healing clock of rate 1. Hence, within a time frame of length $t$, it has a probability of $1-\eulerE^{-t}$ to heal. Let $X$ be the number of initially infected leaves (at time $t'$) that heal until time $t' + t$. Then $X$ is dominated by a binomially distributed random variable $Y \sim \mathrm{Bin}(d\infectionRate \numberOfVertices,1-\eulerE^{-t})$. Using Chernoff bounds and $\frac{t}{2} \leq 1 - \eulerE^{-t}\leq t$, we get
  \begin{align*}
    \Pr{X \geq 2td \infectionRate \numberOfVertices} & \leq \Pr{Y \geq 2td \infectionRate \numberOfVertices}    \\
    & \leq \Pr{Y \geq (1+1) (1- \eulerE^{-t}) d\infectionRate \numberOfVertices}    \\
    & \leq \eulerE^{-\frac{(1- \eulerE^{-t}) d\infectionRate \numberOfVertices}{3}} \\
    & \leq \eulerE^{-\frac{td \infectionRate \numberOfVertices}{6}}.\qedhere
  \end{align*}
\end{proof}

Next, we lower-bound the probability of infecting the center in a time interval of length $t$ when there are more than $d\infectionRate \numberOfVertices /2$ infected leaves.

\begin{restatable}{lemma}{infectingCenter}
  \label{lem:infectingCenter}
  Let $G$ be a star with $\numberOfVertices$ leaves. Let \contactProcess be a \placy process on $G$ with infection rate~$\infectionRate\leq 1$ with $\infectionRate \in \omega(n^{-1})$ and with constant resistance decay rate \deimmunizationRate. Let $d \in \R_{>0}$ be a constant and let $t' \in \R_{\geq 0}$ and $t \in [0,1]$ such that the center healed at time $t'$ and there are at least $d\infectionRate \numberOfVertices/2$ infected leaves at all times in $[t',t'+t]$. Then the probability~$p$ that the center does not get infected until time $t' +t$ is at most $\eulerE^{-d\infectionRate^2\numberOfVertices t^2/32}$.
\end{restatable}

\begin{proof}
  We consider only infections by one of the $d\infectionRate \numberOfVertices/2$ leaves that stay infected the entire time and only infections in the interval $[t'+t/2,t'+t]$ as there the center has a resistance of at most $\eulerE^{-t/2}$. Consider one of those leaves. As its edge to the center triggers at rate $\infectionRate$, there is a probability of $1- \eulerE^{-\infectionRate t /2} \geq \infectionRate t /4$ that this edge triggers in the time interval $[t'+t/2,t'+t]$ and a probability of at least $1- \eulerE^{-t/2} \geq t/4$ of that infection to be successful. Hence, each infected leaf has a probability of at least $\infectionRate t^2/16$ to infect the center. The center can only not get infected if none of the $d\infectionRate \numberOfVertices/2$ considered leaves infects it. As all of those infection attempts are independent, we get

  \begin{align*}
    p & \leq (1-\infectionRate t^2/16)^{d\infectionRate \numberOfVertices/2} \\
      & \leq \eulerE^{-d\infectionRate^2\numberOfVertices t^2/32}.\qedhere
  \end{align*}
\end{proof}

We combine \Cref{lem:healingLeaves,lem:infectingCenter} to show in \Cref{thm:cSIRSlower} a long expected survival time of the \placy process and therefore also for the \cSIRS process, due to \Cref{obs:PLACY}.

\cSIRSlower*
\begin{proof}
  We show the statement for a \placy process. By \Cref{obs:PLACY}, it then follows that the \cSIRS process has the same expected survival time. The idea is to show, for a sufficiently large constant $D \in \R_{\geq 1}$, that the infection survives at least $t \coloneqq D \eulerE^{\numberOfVertices^{2\varepsilon/3}}$ time with some constant probability. To do so, we argue that the infection reaches a state with enough infected leaves with constant probability and then likely stays there longer than $t$.

  By \Cref{lem:endProbability}, using its notation, there is a constant probability for the infection to reach $d \infectionRate \numberOfVertices$ infected leaves before the center heals for some sufficiently small $d\in \R_{>0}$. Let $t' = \numberOfVertices^{-2\varepsilon/3}$. Conditioned on the infection reaching $d \infectionRate \numberOfVertices$ infected leaves, by the pigeonhole principle, one of the following two events has to happen for the infection to die out:
  \begin{enumerate}
    \item Between two center-healing events with less than $d \infectionRate \numberOfVertices$ infected leaves, the number of infected leaves drops by more than $4t'd\infectionRate \numberOfVertices$.
    \item There are $(t')^{-1}/8$ center-healing events in a row that all have less than $d \infectionRate \numberOfVertices$ infected leaves.
  \end{enumerate}
  In fact, by the pigeonhole principle, in order to drop from $d \infectionRate \numberOfVertices$ infected leaves down to $d \infectionRate \numberOfVertices/2$, one of the two events above has to happen. Thus, until any of these two events happened, the number of infected leaves is always above $d \infectionRate \numberOfVertices/2$.

  By \Cref{lem:infectingCenter}, the probability of a center-healthy phase to last longer than $t'$ is in $\bigO{\eulerE^{-d\infectionRate^2\numberOfVertices (t')^2/32}}[\big] = \bigO{\eulerE^{-\numberOfVertices^{2 \varepsilon} \numberOfVertices^{(-2\varepsilon/3) \cdot 2}}}[\big] = \bigO{\eulerE^{-\numberOfVertices^{2\varepsilon/3}}}[\big]$. Hence, with constant probability, no center-healthy phase until time~$t$ is longer than~$t'$. Conditioned on that, by \Cref{lem:healingLeaves} the probability of event 1 happening in a phase is in $\bigO{\eulerE^{-\frac{(t') d\infectionRate \numberOfVertices}{6}}}[\big] = \bigO{\eulerE^{-\numberOfVertices^{1/2 + \varepsilon/3}}}[\big]$. Hence, there is a constant probability of event 1 not happening until time $t$. Note that \Cref{lem:healingLeaves} does not consider how much the number of infected leaves drops in the center-infected phase. However, while below $d \infectionRate \numberOfVertices$ infected leaves, vertices infect much faster than they heal, so it is even less likely that the number of infected leaves drops in this phase (see the proof of \Cref{lem:susInfPhase}). We doubled the number by which the number of infected leaves is allowed to drop in event 1 to account for the center-infected phases as well.

  To bound the probability of event 2 happening, we aim to apply \Cref{lem:endProbability}. For it to be applicable, we need that the number of not-recovered leaves is always at least $2 \frac{\decayRate}{4(2+\decayRate)} \numberOfVertices$, which, by \Cref{lem:recoveredBound} and the discussion following it, has a constant probability of being the case until time~$t$. By \Cref{lem:endProbability}, each center-healthy phase has a constant probability of increasing the number of infected leaves above $d \infectionRate \numberOfVertices$. For event 2 to occur, this has to not happen $(t')^{-1}/8$ times in a row, which has a probability of $\bigO{\eulerE^{-(t')^{-1}}}[\big] = \bigO{\eulerE^{-\numberOfVertices^{2\varepsilon/3}}}[\big]$ for every $(t')^{-1}/8$ consecutive phases. Hence there is a constant probability of event 2 not happening until time~$t$, which concludes the proof.
\end{proof}

\section{The \cSIRS Process on Expanders}

\cite{FrGoKlKrPa2024} show that the SIRS process survives exponentially long on graphs with expanding subgraphs once the infection rate is above a certain threshold, which is identical to the threshold in the SIS process. With minor modifications, this proof translates to a proof of exponential survival time for the \cSIRS process, giving us the following result.

\begin{restatable}{corollary}{expander}
  \label{cor:expander}
  Let $G$ be a graph, and let $G'$ be a subgraph of $G$ that is an $(\numberOfVertices,(1 \pm \degreeGap)\degree,\spectralGap)$-expander (see \cite{FrGoKlKrPa2024}). Let $\degree\rightarrow\infty$ and $\spectralGap,\degreeGap\rightarrow0$ as $\numberOfVertices\rightarrow\infty$. Let \contactProcess be the \cSIRS process on~$G$ with infection rate~$\infectionRate$ and with constant resistance decay rate \deimmunizationRate. Furthermore, let \contactProcess start with at least one infected vertex in $G'$. Last, let~\contactProcessProjection be the projection of \contactProcess onto~$G'$, and let~$T$ be the survival time of~\contactProcessProjection.
  If $\infectionRate \geq \frac{\infectionConstant}{\degree}$ for a constant $\infectionConstant \in \R_{>1}$, then for sufficiently large \numberOfVertices, it holds that $\E{T} = 2^{\bigOmega{\numberOfVertices}}$.
\end{restatable}

\begin{proof}
  In the existing proof on the SIRS process, \cite{FrGoKlKrPa2024} show the super-polynomial survival time in two steps. In the first step, they show that the process reaches a state with a constant fraction of infected vertices with a probability of at least $\frac{1}{\numberOfVertices+2}$. In the second step they proof that the process stays at states with enough infected vertices sufficiently long. For both steps they define a potential function on the state space and show that the process has a positive (or negative) drift, i.e. the value of the potential function increases (or decreases) in expectation in all relevant states. Then they use tools for processes with positive drift to derive the results. We use the property that a \placy process is just a SIRS process with an extra rate to infect recovered vertices to show that the same drift occurs in our process.

  For the first part of their proof \cite[Lemma 1.3]{FrGoKlKrPa2024}, they use a potential function $\potentialIMinusR{\timeDiscrete} = \infectedDiscrete{\timeDiscrete} - \varepsilon_H\recoveredDiscrete{\timeDiscrete}$. The extra rate to infect vertices in the \placy process just increases the drift on this potential, keeping the positive drift. The rest of the arguments in the proof are not affected.

  For the second part, they use a more complex potential with a negative drift in the relevant regions. The proof already works for expanding subgraphs and therefore has to take into account that there might be an arbitrarily large rate at which vertices in the subgraph get infected by vertices outside of the subgraph. The proof accounts for that in \cite[Lemma 4.4]{FrGoKlKrPa2024}. The lemma upper bounds the drift of the defined potential by a term that is independent of that extra infection rate. In the \placy process, the same lemma statement still holds. Just like the extra infection rate, the extra rate of infecting recovered vertices decreases the drift, so any upper bound of it still holds. The rest of the proofs only use this upper bound, so they work exactly the same way as in the SIRS process.
\end{proof}

\section{Conclusion and Future Work}

We study the impact of gradually declining immunity on the SIRS process. To this end, we define the \cSIRS process, whose expected degree of immunity at all points in time is the same as in the original SIRS process. We mathematically rigorously analyze the expected survival time of the \cSIRS process on stars and expander graphs, as well as the expected survival time of the SIRS process on stars.
This allows us to compare these two processes more precisely to each other and also to the well known SIS process, which features no immunity. We prove that the survival time of the \cSIRS process becomes super-polynomial at the same thresholds as the SIS process on both of these graph classes. This stands in contrast to the SIRS process, for which we prove an almost tight polynomial expected survival time on stars for any choice of the infection rate of the process.
Since our lower bounds carry over to graphs that contain stars as subgraphs, such as scale-free networks, they show on a huge variety of networks that gradually declining immunity is far more similar to no immunity at all than to temporary full immunity.

While we show that the \cSIRS process exhibits a super-polynomial expected survival time for the same infection rates as the SIS process, our lower bounds for the \cSIRS process are strictly smaller than the known lower bounds for the SIS process.
This can hint at a potential advantage of gradually declining immunity, albeit just for the super-polynomial regime of the expected survival time.
In order to make this point formal, it would be required to prove upper bounds on the expected survival time of the \cSIRS process in this regime.
Moreover, it would be interesting to investigate whether there are graph classes on which the potential impact of declining immunity is more noticeable than on stars. While we looked at the behavior of the process on stars and cliques, it is still open what happens on graphs that contain both.

Another interesting direction for future work is to consider other functions for the declining resistance than the current choice, which declines exponentially with a rate constant in the number of vertices.
Studying such functions could provide important insights into when immunity starts to have a meaningful impact in decreasing the expected survival time.

Last, our bounds for the SIRS process assume a deimmunization rate~$\deimmunizationRate$ that is independent of the network size.
As this rate approaches zero, the SIRS process should approach the SIS process.
Hence, the effect of full immunity also vanishes.
However, it is not known yet for which exact values of~$\deimmunizationRate$ this is the case.

\printbibliography

@article{berger2005spread,
  title     = {On the spread of viruses on the internet},
  author    = {Berger, Noam and Borgs, Christian and Chayes, Jennifer T. and Saberi, Amin},
  year      = {2005},
  journal   = {Symposium on Discrete Algorithms (SODA)},
  publisher = {Society for Industrial and Applied Mathematics},
  pages     = {301–310},
  url       = {https://dl.acm.org/doi/10.5555/1070432.1070475},
}

@inproceedings{ganesh2005effect,
  title     = {The effect of network topology on the spread of epidemics},
  author    = {Ganesh, Ayalvadi and Massouli{\'e}, Laurent and Towsley, Don},
  booktitle = {International Conference on Computer Communications (INFOCOM)},
  pages     = {1455--1466},
  year      = {2005},
  doi       = {10.1109/INFCOM.2005.1498374},
}

@book{feller1957introduction,
  title     = {An Introduction to Probability Theory and its Applications},
  edition   = {3},
  volume    = {1},
  author    = {Feller, William},
  year      = {1968},
  publisher = {John Wiley \& Sons},
  isbn      = {978-0-471-25708-0},
}

@article{NamNS22SISinfinite,
  title   = {Critical value asymptotics for the contact process on random graphs},
  author  = {Danny Nam and Oanh Nguyen and Allan Sly},
  journal = {Transactions of the American Mathematical Society},
  volume  = {375},
  number  = {6},
  pages   = {3899--3967},
  year    = {2022},
  doi     = {10.1090/tran/8399},
}

@article{BorgsCGS10Antidote,
  author  = {Christian Borgs and Jennifer Chayes and Ayalvadi Ganesh and Amin Saberi},
  title   = {How to distribute antidote to control epidemics},
  journal = {Random Structures \& Algorithms},
  volume  = {37},
  number  = {2},
  pages   = {204-222},
  year    = {2010},
  doi     = {10.1002/rsa.20315},
}

@article{Pastor-SatorrasCVMV15Survey,
  title   = {Epidemic processes in complex networks},
  author  = {Romualdo Pastor-Satorras and Claudio Castellano and Piet Van Mieghem and Alessandro Vespignani},
  journal = {Reviews of Modern Physics},
  volume  = {87},
  number  = {3},
  pages   = {925--979},
  year    = {2015},
  doi     = {10.1103/RevModPhys.87.925},
}

@article{Wang_2017,
  doi       = {10.1088/1742-5468/aa58a6},
  year      = {2017},
  publisher = {{IOP} Publishing},
  volume    = {2017},
  number    = {2},
  pages     = {1--26},
  author    = {Yi Wang and Jinde Cao and Ahmed Alsaedi and Tasawar Hayat},
  title     = {The spreading dynamics of sexually transmitted diseases with birth and death on heterogeneous networks},
  journal   = {Journal of Statistical Mechanics: Theory and Experiment},
}

@article{PhysRevLett.86.2909,
  title   = {Small world effect in an epidemiological model},
  author  = {Kuperman, Marcelo and Abramson, Guillermo},
  journal = {Physical Review Letters},
  volume  = {86},
  issue   = {13},
  pages   = {2909--2912},
  year    = {2001},
  doi     = {10.1103/PhysRevLett.86.2909},
}

@article{Bancal10,
  author  = {Bancal, Jean-Daniel and Pastor-Satorras, Romualdo},
  title   = {Steady-state dynamics of the forest fire model on complex networks},
  journal = {The European Physical Journal B (EPJB)},
  year    = {2010},
  volume  = {76},
  number  = {1},
  pages   = {109-121},
  doi     = {10.1140/epjb/e2010-00165-7},
}

@article{ferreira2016collective,
  title   = {Collective versus hub activation of epidemic phases on networks},
  author  = {Ferreira, Silvio C. and Sander, Renan S. and Pastor-Satorras, Romualdo},
  journal = {Physical Review E},
  volume  = {93},
  number  = {3},
  pages   = {032314},
  year    = {2016},
  doi     = {10.1103/PhysRevE.93.032314},
}

@article{prakash2012threshold,
  title     = {Threshold conditions for arbitrary cascade models on arbitrary networks},
  author    = {Prakash, B. Aditya and Chakrabarti, Deepayan and Valler, Nicholas C and Faloutsos, Michalis and Faloutsos, Christos},
  journal   = {Knowledge and information systems},
  volume    = {33},
  pages     = {549--575},
  year      = {2012},
  publisher = {Springer},
  doi       = {10.1007/s10115-012-0520-y},
}

@inproceedings{kempe2003maximizing,
  title     = {Maximizing the spread of influence through a social network},
  author    = {Kempe, David and Kleinberg, Jon and Tardos, {\'E}va},
  booktitle = {Proceedings of the ninth ACM SIGKDD international conference on Knowledge discovery and data mining},
  pages     = {137--146},
  year      = {2003},
}

@inproceedings{leskovec2007cost,
  title     = {Cost-effective outbreak detection in networks},
  author    = {Leskovec, Jure and Krause, Andreas and Guestrin, Carlos and Faloutsos, Christos and VanBriesen, Jeanne and Glance, Natalie},
  booktitle = {Proceedings of the 13th ACM SIGKDD international conference on Knowledge discovery and data mining},
  pages     = {420--429},
  year      = {2007},
}

@inproceedings{SunCM2023SocialInfluenceMaximization,
  title     = {Social Influence-Maximizing Group Recommendation},
  author    = {Sun, Yangke and Cautis, Bogdan and Maniu, Silviu},
  booktitle = {Proceedings of the International AAAI Conference on Web and Social Media},
  pages     = {820-831},
  year      = {2023},
}

@inproceedings{SunRZLY2022HypergraphAttentionNetwork,
  title     = {MS-HGAT: Memory-Enhanced Sequential Hypergraph Attention Network for Information Diffusion Prediction},
  author    = {Sun, Ling and Rao, Yuan and Zhang, Xiangbo and Lan, Yuqian and Yu, Shuanghe},
  booktitle = {Proceedings of the AAAI Conference on Artificial Intelligence},
  year      = {2022},
  pages     = {4156-4164},
}

@inproceedings{JiangRF2023PoliticalLearning,
  title     = {Retweet-BERT: Political Leaning Detection Using Language Features and Information Diffusion on Social Networks},
  author    = {Jiang, Julie and Ren, Xiang and Ferrara, Emilio},
  booktitle = {Proceedings of the International AAAI Conference on Web and Social Media},
  year      = {2023},
  pages     = {459-469},
}

@inproceedings{SharmaHSL2021FakeNews,
  title     = {Network Inference from a Mixture of Diffusion Models for Fake News Mitigation},
  author    = {Sharma, Karishma and He, Xinran and Seo, Sungyong and Liu, Yan},
  booktitle = {Proceedings of the International AAAI Conference on Web and Social Media},
  year      = {2021},
  pages     = {668-679},
}

@inproceedings{LiuRGCXTL2023DOVID19,
  title     = {Human Mobility Modeling during the COVID-19 Pandemic via Deep Graph Diffusion Infomax},
  author    = {Liu, Yang and Rong, Yu and Guo, Zhuoning and Chen, Nuo and Xu, Tingyang and Tsung, Fugee and Li, Jia},
  booktitle = {Proceedings of the AAAI Conference on Artificial Intelligence},
  year      = {2023},
  pages     = {14347-14355},
}

@article{RazaqueRKAR2022Survey,
  title   = {State-of-art review of information diffusion models and their impact on social network vulnerabilities},
  author  = {Abdul Razaque and Syed Rizvi and Meer Jaro Khan and Muder Almiani and Amer Al Rahayfeh},
  journal = {Journal of King Saud University - Computer and Information Sciences},
  volume  = {34},
  number  = {1},
  pages   = {1275-1294},
  year    = {2022},
}

@inproceedings{friedrich2024irrelevance,
  title     = {The Irrelevance of Influencers: Information Diffusion with Re-Activation and Immunity Lasts Exponentially Long on Social Network Models},
  author    = {Friedrich, Tobias and G{\"o}bel, Andreas and Klodt, Nicolas and Krejca, Martin S and Pappik, Marcus},
  booktitle = {38th Annual AAAI Conference on Artificial Intelligence},
  pages     = {17389--17397},
  year      = {2024},
}

@article{FrGoKlKrPa2024,
  author   = {Tobias Friedrich and Andreas G\"obel and Nicolas Klodt and Martin S. Krejca and Marcus Pappik},
  title    = {Analysis of the survival time of the SIRS process via expansion},
  journal  = {Electronic Journal of Probability},
  year     = {2024},
  volume   = {29},
  pno      = {83},
  pages    = {1-29},
  issn     = {1083-6489},
  doi      = {10.1214/24-EJP1140},
  sici     = {1083-6489(2024)29:83<1:AOTSTO>2.0.CO;2-L},
}

@article{tricomi1951asymptotic,
  title   = {The asymptotic expansion of a ratio of gamma functions},
  author  = {Tricomi, Francesco Giacomo and Erd{\'e}lyi, Arthur},
  journal = {Pacific Journal of Mathematics},
  volume  = {1},
  number  = {1},
  pages   = {133--142},
  year    = {1951},
}

@article{vaccines10010064,
  author         = {Israel, Ariel and Shenhar, Yotam and Green, Ilan and Merzon, Eugene and Golan-Cohen, Avivit and Schäffer, Alejandro A. and Ruppin, Eytan and Vinker, Shlomo and Magen, Eli},
  title          = {Large-Scale Study of Antibody Titer Decay following BNT162b2 mRNA Vaccine or SARS-CoV-2 Infection},
  journal        = {Vaccines},
  volume         = {10},
  year           = {2022},
  number         = {1},
  article-number = {64},
  pages          = {64},
  url            = {https://www.mdpi.com/2076-393X/10/1/64},
  pubmedid       = {35062724},
  issn           = {2076-393X},
  doi            = {10.3390/vaccines10010064},
}

@article{gaebler2021evolution,
  title     = {Evolution of antibody immunity to SARS-CoV-2},
  author    = {Gaebler, Christian and Wang, Zijun and Lorenzi, Julio CC and Muecksch, Frauke and Finkin, Shlomo and Tokuyama, Minami and Cho, Alice and Jankovic, Mila and Schaefer-Babajew, Dennis and Oliveira, Thiago Y and others},
  journal   = {Nature},
  volume    = {591},
  number    = {7851},
  pages     = {639--644},
  year      = {2021},
  publisher = {Nature Publishing Group},
}

@article{wheatley2021evolution,
  title     = {Evolution of immune responses to SARS-CoV-2 in mild-moderate COVID-19},
  author    = {Wheatley, Adam K and Juno, Jennifer A and Wang, Jing J and Selva, Kevin J and Reynaldi, Arnold and Tan, Hyon-Xhi and Lee, Wen Shi and Wragg, Kathleen M and Kelly, Hannah G and Esterbauer, Robyn and others},
  journal   = {Nature communications},
  volume    = {12},
  number    = {1},
  pages     = {1162},
  year      = {2021},
  publisher = {Nature Publishing Group UK London},
}

@article{PhysRevE.86.041125,
  title     = {Epidemic thresholds of the susceptible-infected-susceptible model on networks: A comparison of numerical and theoretical results},
  author    = {Ferreira, Silvio C. and Castellano, Claudio and Pastor-Satorras, Romualdo},
  journal   = {Physical Review E},
  volume    = {86},
  issue     = {4},
  pages     = {041125},
  numpages  = {8},
  year      = {2012},
  month     = {Oct},
  publisher = {American Physical Society},
  doi       = {10.1103/PhysRevE.86.041125},
  url       = {https://link.aps.org/doi/10.1103/PhysRevE.86.041125},
}

@article{ZhangLZLZZ16Overview,
  author  = {Zi-Ke Zhang and Chuang Liu and Xiu-Xiu Zhan and Xin Lu and Chu-Xu Zhang and Yi-Cheng Zhang},
  title   = {Dynamics of information diffusion and its applications on complex networks},
  journal = {Physics Reports},
  volume  = {651},
  pages   = {1-34},
  year    = {2016},
  doi     = {https://doi.org/10.1016/j.physrep.2016.07.002},
}

@article{lam2024optimalboundsurvivaltime,
  title={Optimal bound for survival time of the SIRS process on star graphs},
  author={Lam, Phuc and Nguyen, Oanh and Yang, Iris},
  journal={arXiv preprint arXiv:2412.21138},
  year={2024}
}

@article{sanderson2021covid,
  title={COVID vaccines protect against Delta, but their effectiveness wanes.},
  author={Sanderson, Katharine},
  journal={Nature},
  year={2021}
}

@article{watve2024epidemiology,
  title={Epidemiology: Gray immunity model gives qualitatively different predictions},
  author={Watve, Milind and Bhisikar, Himanshu and Kharate, Rohini and Bajpai, Srashti},
  journal={Journal of Biosciences},
  volume={49},
  number={1},
  pages={10},
  year={2024},
  publisher={Springer}
}

@article{fouchet2008visiting,
  title={Visiting sick people: is it really detrimental to our health?},
  author={Fouchet, David and O'Brien, John and Pontier, Dominique},
  journal={Plos one},
  volume={3},
  number={6},
  pages={e2299},
  year={2008},
  publisher={Public Library of Science San Francisco, USA}
}

@article{aguas2006pertussis,
  title={Pertussis: increasing disease as a consequence of reducing transmission},
  author={{\'A}guas, Ricardo and Gon{\c{c}}alves, Guilherme and Gomes, M Gabriela M},
  journal={The Lancet infectious diseases},
  volume={6},
  number={2},
  pages={112--117},
  year={2006},
  publisher={Elsevier}
}

\end{document}